\documentclass[11pt,reqno]{amsart}
\usepackage{amscd,amsmath,amsopn,amssymb,amsthm,multicol}
\usepackage{tikz,subdepth,anysize,verbatim,ifthen,xargs,colortbl,float}
\usepackage{longtable,mathtools}
\usepackage[english]{babel}
\usepackage[utf8]{inputenc}
\everymath=\expandafter{\the\everymath\displaystyle}
\usetikzlibrary{decorations.markings,arrows,arrows.meta,bending,calc}
\tikzset{>={Stealth[scale=1.5, bend]}}
\theoremstyle{plain}
\newtheorem{theorem}{Theorem}
\newtheorem{prop}{Proposition}

\newtheorem{cor}{Corollary}

\theoremstyle{definition}
\newtheorem{definition}{Definition}

\newtheorem{remark}{Remark}

\newcommand\com[1]{}
\newcommand\C{{\mathbb C}}
\newcommand\D{\mathcal{D}}
\newcommand\g{{\mathfrak g}}
\newcommand\h{\mathfrak{h}}
\newcommand\fp{{\frak p}}
\newcommand\m{{\frak m}}

\newcommand\op[1]{\mathop{\rm #1}\nolimits}
\newcommand\p{\partial}

\newcommand\R{{\mathbb R}}
\newcommand\Z{{\mathbb Z}}
\newcommand\chh{\raisebox{2pt}{$\chi$}}
\newcommand\cch{{\footnotesize\raisebox{1pt}{$\chi$}}}
\arrayrulecolor{gray!50}

\begin{document}

\title[Realization of $G(3)$ and $F(4)$ supersymmetries]%
{Realization of Lie superalgebras $G(3)$ and $F(4)$ \\ as symmetries of 
supergeometries}

\author{Boris Kruglikov}
\address{Department of Mathematics and Statistics, UiT the Arctic University of Norway, Troms\o\ 9037, Norway.
\ E-mail: {\tt boris.kruglikov@uit.no}. }

\author{Andreu Llabr\'{e}s}
\address{Department of Mathematics and Statistics, UiT the Arctic University of Norway, Troms\o\ 9037, Norway.
\ E-mail: {\tt andreu.llabres@uit.no}. }

 \begin{abstract}
For every parabolic subgroup $P$ of a Lie supergroup $G$ the homogeneous superspace $G/P$
carries a $G$-invariant supergeometry. We address the quesiton whether $\g=\op{Lie}(G)$ is the maximal
symmetry of this supergeometry in the case of exceptional Lie superalgebras $G(3)$ and $F(4)$.
Our approach is to consider the negatively graded Lie superalgebras for every choice of parabolic, and to compute
the Tanaka-Weisfeiler prolongations, with reduction of the structure group when required (2 resp 3 cases),
thus realizing $G(3)$ and $F(4)$ as symmetries of supergeometries. This gives 19 inequivalent
$G(3)$-supergeometries and 55 inequivalent $F(4)$-supergeometries,
in majority of cases (17 resp 52 cases) those being encoded as vector superdistributions.
We describe those supergeometries and realize supersymmetry explicitly in some cases.
 \end{abstract}

\maketitle

\section{Introduction}\label{S1}

While classical simple Lie algebras appeared as symmetries of certain bilinear or volume forms, the
five exceptionals were first introduced abstractly. The simplest of these, the 14-dimensional Lie algebra $G_2$,
discovered in 1887, was realized as the symmetry algebra of two different Klein geometries in 1893 by
E. Cartan and F. Engel, in two successive papers in the same issue of Comptes Rendus \cite{Ca,En}.
Realizations of $F_4,E_6,E_7,E_8$ are due to A. Borel, H. Freudental, J. Tits and others
in early 1950s, see \cite{Fr,Ti}. 
Recently there was a universal approach to geometric realization of all simple Lie algebras \cite{The}.

It should be noted the above realizations of Cartan and Engel correspond to homogeneous geometries
$G_2/P_1$, $G_2/P_{12}$ and $G_2/P_2$, where $P_\cch$ are parabolic subgroups encoded by
subsets of simple roots, marked by crosses $\chh$ on the Dynkin diagram of $G_2$.
Geometric structures on those are the maximally symmetric $(2,3,5)$ distribution
and the field of rational normal curves on a contact structure of a 5-dimensional manifold.
(There are other realizations of $G_2$ as symmetries: octonion in 8D, cross product on $\R^7$,
almost complex structure on $S^6$; all these are related to the quotient $G_2/SU_3$ in the real compact version,
and there is a split counter-part.\footnote{In this paper we work exclusively over $\C$, but the results
have certain real analogs.})
A vital part of this statement is not only that $G_2$ preserves the indicated
structures, but also that it is the entire automorphism group. Locally we get the same result
for the symmetry algebra.

In \cite{KST} a similar realization was achieved for the exceptional Lie superalgebra $G(3)$ and in \cite{ST} for $F(4)$.
Each was realized twice, and each realization corresponds to generalized flag supervariety $G/P$ for
a particular choice of parabolic subgroup $P\subset G$. On the level of Lie algebras, $\fp\subset\g$
is also obtained by marking some of the nodes on the Dynkin diagram, which corresponds to a choice of
$\Z$-grading; however in the super-case there are non-conjugate choices of simple roots (hence different Dynkin diagrams)
related by odd reflections. Thus there are 4 Dynkin diagrams for $G(3)$ and 6 such for $F(4)$.
Altogether this gives 19 inequivalent parabolic subalgebras $\fp$ of $G(3)$ and 55 such subalgebras of $F(4)$.

The goal of this paper is to realize $G(3)$ and $F(4)$ as symmetries of Klein supergeometries,
corresponding to all those cases. Note that if $\g=\g_{-\nu}\oplus\dots\oplus\g_0\oplus\dots\oplus\g_\nu$ is
the $\Z$-grading corresponding to the choice of $\fp=\g_0\oplus\dots\oplus\g_\nu$, and
$\m=\g_-=\g_{-\nu}\oplus\dots\oplus\g_{-1}$ is the nilradical of the opposite parabolic, then
$\exp(\m)\subset G/P$ is an open subsupermanifold. To prove that $\g=\op{Lie}(G)$ (in our case $G(3)$ or $F(4)$)
is the entire symmetry superalgebra we consider the Tanaka-Weisfeiler prolongation of $\m$:
$\g=\op{pr}(\m)$ is the maximal $\Z$-graded superalgebra with $\g_-=\m$ and having no centralizer of $\g_{-1}$
in $\g_{\ge0}$. Similarly one defines the Tanaka-Weisfeiler prolongation $\g=\op{pr}(\m,\g_0)$ of $\m\oplus\g_0$.

We refer to \cite{L,CCF} for the basics on supergeometries and to \cite{K,Va} for Lie superalgebras.
Let us denote $\m^{\Xi}_{\cch}$, $\fp^{\Xi}_{\cch}$ the negative and non-negative (parabolic) parts
of the grading corresponding to Dynkin diagram $\Xi$ and parabolic\footnote{Parabolic subalgebra is
generated by a Borel subalgebra 
and negative root vectors $e_{-\alpha}$ for non-crossed simple roots $\alpha$.} $\chh$, so that the grading is
 \begin{equation}\label{grading}
\g=
\underbrace{\g^{\Xi,\cch}_{-\mu}\oplus\dots\oplus\g^{\Xi,\cch}_{-1}}_{\m^{\Xi}_{\cch}}
\oplus\underbrace{\g^{\Xi,\cch}_0\oplus\g^{\Xi,\cch}_{1}\dots\oplus \g^{\Xi,\cch}_{\mu}}_{\fp^{\Xi}_{\cch}}.
 \end{equation}
We will specify details, also on numeration of systems of simple roots and parabolics, in the following sections.
Irreducible case (or $|1|$-grading) corresponds to $\mu=1$, while contact grading is characterized by
$\mu=2$ and $\dim\g_{\pm2}=1$. One of our main results is:

 \begin{theorem}\label{main}
Let a simple Lie superalgebra $\g$ be either $G(3)$ or $F(4)$. Then for any choice of parabolic we have
$\op{pr}(\mathfrak{m}^{\Xi}_{\cch})=\g$
in all cases except for two contact gradings for $G(3)$ and one irreducible and two contact gradings
for $F(4)$, in which cases we have $\op{pr}(\mathfrak{m}^{\Xi}_{\cch},\g_0^{\Xi,\cch})=\g$.
 \end{theorem}

The indicated supergeometries are described in Sections \ref{S3}-\ref{S4}. We recall in details the definition of
the Tanaka-Weisfeiler prolongation \cite{Ta} in Section \ref{S2}, before proving the above theorem in Section \ref{S5}.
The proof is entirely computer based, and we will explain the reason for this approach in that section.
Note that by the very construction, Lie superalgebra $\g$ acts by symmetries of the distributions/induced geometries
on $G/P^{\Xi}_{\cch}$. According to \cite{KST2} the symmetry is majorized by $\op{pr}(\mathfrak{m}^{\Xi}_{\cch})$
or $\op{pr}(\mathfrak{m}^{\Xi}_{\cch},\g_0^{\Xi,\cch})$ respectively. Hence we get:

 \begin{cor}\label{cor}
There are 19 distinct supergeometries with symmetry $G(3)$ and 55 distinct supergeometries with symmetry $F(4)$.
They are supported on generalized flag supervarieties $G/P$ for various parabolics $P$.
In the case of $G(3)$, 17 of those supergeometries are supervector distributions, while 2 are
reductions of the structure group for the odd and mixed parity contact structures on supermanifolds
of dimensions $(1|7)$ and $(5|4)$ respectively.
In the case of $F(4)$, 52 of those supergeometries are supervector distributions, while 1 is a $G_0$-structure
on a manifold of dimension $(6|4)$ and 2 are
reductions of the structure group for the odd and mixed parity contact structures on supermanifolds
of dimensions $(1|8)$ and $(7|4)$ respectively.
 \end{cor}

The structure of superdistributions $\D$ will be clear from the list of different positive roots given in Sections \ref{S3}-\ref{S4};
growth vectors of these $\D$ are provided in the appendix \ref{S9} (it is not difficult also to provide
the description of the corresponding group $G_0$ as was done for all parabolics of $G(3)$ in \cite{KST} and
for maximal parabolics of $F(4)$ in \cite{ST}, but we will not need it). It can be also extracted from
the supplementary \textsc{Maple} file.

Let us note that global vector fields on generalized flag supervarieties $G/P$ were computed for several
classical Lie superalgebras $\g$ (but not for exceptionals) in \cite{OS,V1,V2} with the conclusion that
they are fundamental (equal to $\g$) in many cases. Relation of this global problem to the local problem,
discussed in the current paper, is via the Bott-Borel-Weil theorem \cite{Ko}, which does not hold
in the same strength in the supersetting as in the classical situation, see e.g.\ \cite{Co}.
We will discuss this more in 
Section \ref{S8}.

Finally we ask whether the above supergeometries can be encoded in the language
of differential equations, similar as $G(3)$ and $F(4)$ were realized as supersymmetries of
differential equations in \cite{KST,ST}, leading to super Hilbert-Cartan and exceptionally simple PDEs of loc.\ cit.
In fact, the first realization of $G_2$ as a symmetry of a differential equation is due to E.\,Cartan 1893,
in the form of overdetermined system of PDEs, and later also in the form of ODE
(Hilbert-Cartan equation: implicitly in 1910, explicitly in 1912 and 1914).

In the classical situation such realization is clearly possible, for instance one can encode the equation
for integral curves of the distribution as an underdetermined ODE (so-called Monge equation from control theory). 
The situation in the supercase is more delicate. 
We will show that integral curves may be insufficient to encode the distribution, even though they
can encode the underlying subbundle of the reduced tangent bundle. 
However in some cases we can exploit higher dimensional integral submanifolds,
and also other means to encode symmetry of the geometry via differential equations. 
This is subject of Sections \ref{S6}-\ref{S7},
where we also discuss the structure reductions.


\bigskip

{\bf Acknowledgment.}
The authors thank Andrea Santi for useful discussions.
The research leading to our results has received funding from
the Norwegian Financial Mechanism 2014-2021 (project registration number 2019/34/H/ST1/00636),
the Polish National Science Centre (NCN grant number 2018/29/B/ST1/02583),
and the Tromsø Research Foundation (project “Pure Mathematics in Norway”).
It was also supported by the UiT Aurora project MASCOT.

\section{Vector superdistributions and flag supervarieties}\label{S2}

Let $M=(M_o, \mathcal{A}_M)$ be a supermanifold, i.e.\ $\mathcal{A}_M$ is a $\Z_2$-graded sheaf
of superfunctions on the underlying manifold $M_o$.
A vector superdistribution (or distribution on a supermanifold) is a (graded) $\mathcal{A}_M$-subsheaf
$\D$ of the tangent sheaf $\mathcal{T}M=\op{Der}(\mathcal{A}_M)$ of $M$ that is projective
(locally a direct factor) \cite{Va}.
Any such $\D$ induces a vector subbundle $D=\op{ev}(\D)\subset TM|_{M_o}$ of the reduced tangent bundle, 
defined as $D=\D|_{M_o}=\cup_{x\in M_o}\D|_x$, where $\D|_x$ is the evaluation of $\D$ at $x\in M_o$ 
(this bundle over $M_o$ does not fully determine $\D$).

The weak derived flag of $\D$ is defined as follows:
 \begin{equation}\label{eq:weakderivedflag}
\D^{1}=\D \subset \D^{2}\subset\cdots\subset\D^{k}\subset\cdots\;,\qquad
\D^{k+1} = [\D, \D^{k}].
 \end{equation}
The distribution $\D$ on $M$ is {\em regular\/} if $\D^k$ are vector superdistributions for all $k>0$.
In this case, $D^k=\D^k|_{M_o}=\cup_{x\in M_o}\D^k|_x$ is a vector bundle over $M_o$ for every $k>0$.
The distribution is completely nonholonomic (bracket-generating) if $\D^{\mu}=\mathcal TM$ for some
(minimal) $\mu$, called the depth of $\D$.

Setting $\g_{-k}(x)=(\D^{k}|_x)/(\D^{k-1}|_x)$, the symbol algebra (also called Carnot algebra) at $x\in M_o$ is
$\m_x=\bigoplus_{k<0}\g_k(x)$. Note that $\g_k=0$ for all $k<-\mu$. More generally, one
may consider the stalk $\D_x^k$ of $\D^k$ at $x\in M_o$ as a module over the local ring $(\mathcal{A}_M)_x$
and set $\op{gr}(\mathcal T_xM)=\oplus_{k>0} \op{gr}(\mathcal T_xM)_{-k}$, where
$\op{gr}(\mathcal T_xM)_{-k}=\D_x^k/\D_x^{k-1}$.
This is naturally a graded Lie superalgebra free over $(\mathcal{A}_M)_x$.

 \begin{definition}
A superdistribution $\D$ is strongly regular if there exists a negatively-graded Lie superalgebra
$\m=\bigoplus_{0<k\leq\mu}\m_{-k}$ such that $\op{gr}(\mathcal T_xM)\cong (\mathcal{A}_M)_x\otimes \m$
at any $x\in M_o$, as graded Lie superalgebras over $(\mathcal{A}_M)_x$.
 \end{definition}

The superdistributions arising in this paper are all strongly regular.

 \begin{definition}\label{def prolongation}
The Tanaka-Weisfeiler prolongation of $\m$ is the unique (possibly infinite-dimensional)
$\mathbb Z$-graded Lie superalgebra
 $$\op{pr}(\m)=\bigoplus_{k=-\mu}^{+\infty} \g_k$$
that extends $\m$, is transitive ($[X,\g_{-1}]\neq0$ for $0\neq X\in\g_{\ge0}$) and is maximal with these properties.

If $\g_0\subset\op{der}_0(\m)$ is a Lie subsuperalgebra of grade-preserving derivations of $\m$ then
 $$\op{pr}(\m,\g_0)=\bigoplus_{k=-\mu}^{+\infty} \g_k$$
is the prolongation defined by the same properties but extending $\m\oplus\g_0$.
 \end{definition}

The proof of the existence and uniqueness of this algebraic prolongation, given in the classical case in \cite{Ta}, 
extends verbatim to Lie superalgebras. The prolongation is finite if $\g_k=0$ for some $k\ge0$.

Note that for a graded Lie superalgebra $\g$ with $\m=\g_{<0}$ the property $\op{pr}(\m)=\g$
is equivalent to the equality $H^1(\m,\g)_{\ge0}=0$, while $\op{pr}(\m,\g_0)=\g$
is equivalent to the equality $H^1(\m,\g)_+=0$, see e.g.\ \cite{KST}.
Here we refer to the natural grading of the Chevalley-Eilenberg
cohomology of Lie superalgebra $\m$ with values in the graded $\m$-module $\g$. Namely, the
cohomology in grading $k$ of the following complex
 $$
0\to\g\longrightarrow\m^*\otimes\g\longrightarrow\Lambda^2\m^*\otimes\g\longrightarrow\dots
 $$

For a basic classical Lie superalgebra its gradings $\g=\oplus_{k=-\mu}^{\mu}\g_k$
are bijective (up to conjugation) with parabolic subalgebras $\fp\subset\g$ \cite{K}.
The filtration $\g^{-k}=\oplus_{i\ge-k}\g_i$ is $\fp$-invariant.

Let $G$ be a Lie supergroup with subgroup $P$, having Lie superalgebras $\g\supset\fp$ respectively.
Consider the homogeneous space (generalized flag supervariety) $G/P$ in the sense of \cite{M,Va}.
This superspace possesses a superdistribution $\D$ corresponding to $\g^{-1}$ and
the fundamental property ($\g_{-1}$ is bracket-generating) implies that the associated weak derived flag $\D^k$
corresponds to $\g^{-k}$.

Another, so-called standard model is $\exp(\m)$, which is an open subsupermanifold of $G/P$.
It also contains the vector superdistribution $\D$, and in both cases $\g$ acts as a symmetry algebra.
Our goal is to prove that this is the entire symmetry algebra for the induced geometry.

A relation between the generalized flag supervarieties is the following twistor correspondence:
nested parabolics $\fp\subset\mathfrak{q}\subset\g$ determine gradings
 $$
\g=\m^{\fp}_{-}\oplus\fp_0\oplus\fp_+=\m^{\mathfrak{q}}_{-}\oplus\mathfrak{q}_0\oplus\mathfrak{q}_+
 $$
with $\mathfrak{q}_+\subset\fp_+$ and $\fp_0^{ss}\subset\mathfrak{q}_0^{ss}$ (for semisimple parts).
We get the natural map of supergeometries
 $$
G/P\longrightarrow G/Q,
 $$
where the left and right hand sides are the correspondence and twistor spaces respectively. For $G(3)$ and $F(4)$
flag supervarieties this twistor correspondence is shown on Figures \ref{geometries_diagram} and \ref{F4map} respectively.

In what follows we also use the following observation about Weyl reflection groupoid.
Assume that two system of simple roots are related by an odd reflection \cite{Se} at odd simple root $\alpha_i$,
or equivalently at grey node $i$ of the Dynkin diagrams $\Xi$, $\Xi'$ with nodes $N,N'$, and that the remaining nodes are
premuted by the bijection $z:N\to N'$, $z(i)=i'$. Then for a set $\chh\subset N\setminus\{i\}$ the parabolic
subsuperalgebra $\fp_\cch^\Xi$ in $\g$ is isomorphic to $\fp_{\cch'}^{\Xi'}$ with $\chh'=z(\chh)\subset N\setminus\{i'\}$.
The same is true for the generalized flag supervarieties $M^\Xi_\cch$
(which is the global model $G/P^\Xi_\cch$ or its local version $\exp(\m^\Xi_\cch)$), 
and this is indicated on Figures \ref{geometries_diagram} and \ref{F4map}.

\section{\(G(3)\)-supergeometries}\label{S3}

The Lie superalgebra \(G(3)\) has 4 different root systems up to \(\mathcal{W}\)-equivalence \cite{FSS}, where \(\mathcal{W}\) is the Weyl group generated by all even reflections on even roots. Each of those root systems has 3 simple roots \(\alpha_1\), \(\alpha_2\), \(\alpha_3\). From each of the 4 inequivalent root system, one can obtain the other 3 root systems by odd reflections \cite{Se} (dashed lines on Figure \ref{dynkinG3}) see \cite{KST} for details.

\begin{figure}[H]
\begin{center}
\includegraphics[width=0.6\textwidth]{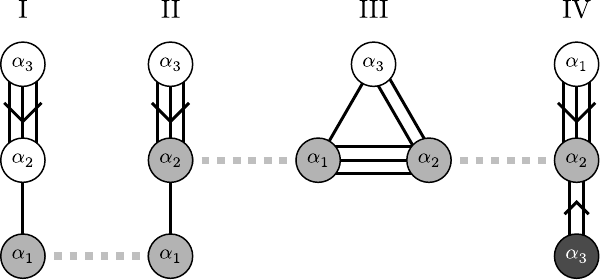}
\caption{\(\mathcal{W}\)-inequivalent Dynkin diagrams of \(G(3)\).}
\label{dynkinG3}
\end{center}
\end{figure}

The 4 Dynkin diagrams are labeled with roman numerals. The Cartan subalgebra \(\mathfrak{h}\) of \(\mathfrak{g}\) is a subalgebra of \(\mathfrak{g}_{\overline 0}\). The \(G(3)\) root system \(\Delta = \Delta_{\overline 0} \cup \Delta_{\overline 1} \subset \mathfrak{h}^*\setminus \{0\}\) is given by
\begin{equation}
\Delta _{\overline 0}= \left\{ \pm 2\delta, \pm \varepsilon_i, \varepsilon_i-\varepsilon_j \right\} \,,\quad\quad \Delta_{\overline 1} = \left\{ \pm \delta, \pm \delta \pm \varepsilon_i \right\}
\label{even_odd_dual_cartan-G3}
\end{equation}
where \(\delta, \varepsilon_1,\varepsilon_2,\varepsilon_3\) are vectors in 3-dimensional space \(\mathfrak{h}^*\) satisfying relation \(\varepsilon_1+\varepsilon_2+\varepsilon_3=0\). 
Their scalar products with respect to the Killing form are: \(\langle \varepsilon_i,\varepsilon_j \rangle =1-3\delta _{ij}\), \(\langle \delta,\delta \rangle =2\), \(\langle \varepsilon_i,\delta \rangle =0\).

For each Dynkin diagram label \(\Xi\in \left\{ \text{I},\text{II},\text{III},\text{IV} \right\} \), the corresponding simple root system \(\Pi_\Xi= \left\{ \alpha_1, \alpha_2, \alpha_3 \right\} \) is defined in table \ref{roots}.
Positive roots are nontrivial  linear combinations of those with non-negative coefficients.

\begin{table}[H]
\begin{center}
\(\begin{array}{|c|cccc|}
\hline
& \text{I} & \text{II} & \text{III} & \text{IV}\\\hline
\alpha_1 & \delta - \varepsilon_1 - \varepsilon_2 & \varepsilon_1+\varepsilon_2-\delta & \varepsilon_2-\delta & \varepsilon_2-\varepsilon_1\\
\alpha_2 & \varepsilon_1 & \delta-\varepsilon_2 & \delta-\varepsilon_1 & \varepsilon_1-\delta\\
\alpha_3 & \varepsilon_2 - \varepsilon_1 & \varepsilon_2-\varepsilon_1 & \varepsilon_1 & \delta\\\hline
\end{array}\)
\end{center}
\caption{\(\Pi_\Xi\) for each \(\Xi\in \left\{ \text{I}, \text{II}, \text{III}, \text{IV} \right\} \) in \(G(3)\).}
\label{roots}
\end{table}

A choice of root system type \(\Xi\) together with a choice of a parabolic subgroup \(P^\Xi_{\cch}\), with \(\chh\in \mathcal{P}(\left\{1,2,3 \right\})\setminus \left\{ \emptyset \right\}  \),  gives one of 19 possible supergeometries
\(G(3)/P^\Xi_\cch\), shown in Figure \ref{geometries_diagram}.

\begin{figure}[H]
\begin{center}
\includegraphics[width=0.85\textwidth]{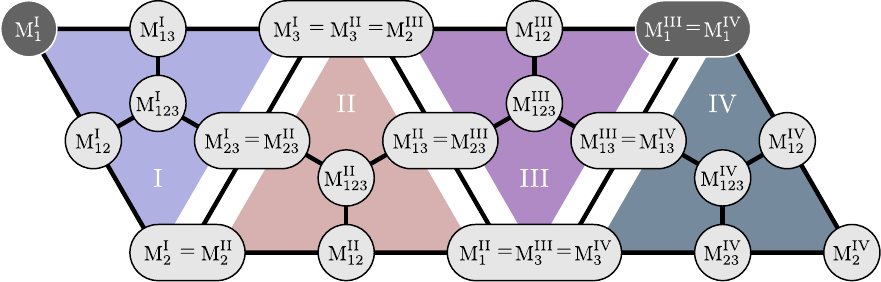}
\end{center}
\caption{Map of \(G(3)\)-supergeometries. Dark nodes correspond to supergeometries that require a \(\mathfrak{g}_0\)-reduction. Edges in the diagram connect supergeometries \(M ^{\Xi}_{\cch_1}\) and \(M ^{\Xi}_{\cch_2}\) such that \(|\chh_1\ominus\chh_2|=1\) (cardinal of symmetric difference is one).}
\label{geometries_diagram}
\end{figure}

Using definitions in table \ref{roots}, the linear combinations of \(\alpha_1, \alpha_2, \alpha_3\) that belong to \(\Delta_{\overline 0}\) and \(\Delta_{\overline 1}\) in \eqref{even_odd_dual_cartan-G3} define even and odd roots respectively for the root system. In Table \ref{negative roots G3}, there are the negative roots, which are 
nontrivial linear combinations of simple roots with non-positive coefficients.

\def\s{0.4}
\def\h{-0.4}
\def\a{0.17}
\def\b{0.05}
\definecolor{colortwo}{rgb}{0.7,0.7,0.9}
\definecolor{colorthree}{rgb}{0.85,0.7,0.7}
\definecolor{colorfour}{rgb}{0.47,0.55,0.62}
\def\czero{white}
\def\cone{gray!20}
\def\ctwo{colortwo}
\def\cthree{colorthree}
\def\cfour{colorfour}
\def\c{\czero}
\newcommand{\numbertocolor}[1]{
\ifthenelse{#1=0}{\def\c{\czero}}{}
\ifthenelse{#1=1}{\def\c{\cone}}{}
\ifthenelse{#1=2}{\def\c{\ctwo}}{}
\ifthenelse{#1=3}{\def\c{\cthree}}{}
\ifthenelse{#1=4}{\def\c{\cfour}}{}
}
\pgfmathsetmacro{\bl}{2*\h-\b}
\def\arraycode{
\numbertocolor{\j}
\ifthenelse{\j>0}{
\foreach \m in {1,...,\j}
\draw[fill=\c] (\i*\s,1*\h) circle ({\a-(\m-1)*\b});}{
\draw[densely dotted, fill=\c] (\i*\s,1*\h) circle (\a);}
\numbertocolor{\k}
\ifthenelse{\k>0}{
\foreach \m in {1,...,\k}
\draw[fill=\c] (\i*\s,2*\h) circle ({\a-(\m-1)*\b});}{
\draw[densely dotted, fill=\c] (\i*\s,2*\h) circle (\a);}
\numbertocolor{\l}
\ifthenelse{\l>0}{
\foreach \m in {1,...,\l}
\draw[fill=\c] (\i*\s,3*\h) circle ({\a-(\m-1)*\b});}{
\draw[densely dotted, fill=\c] (\i*\s,3*\h) circle (\a);}}
\begin{table}[H]
\begin{center}
\(\begin{array}{|c|ccccc|}\hline
\tikz[baseline=-1mm]\node{\(\Delta^{\Xi}_{\rho,-}\)};& \Xi: &\text{I}\hspace{0.2cm}\phantom{i} & \text{II} & \text{III} & \text{IV}\phantom{a} \\\hline & & & & &\\[-2mm]
\tikz[baseline=-1mm]\node[rotate=0]{\(\rho=\overline 0\)};& \begin{array}{c}
\alpha_1 \\[-1mm] \alpha_2 \\[-1mm] \alpha_3
\end{array} &\hspace{-0.3cm}
\begin{tikzpicture}[baseline=\bl cm]
\foreach \i/\j/\k/\l in {0/0/1/0, 1/0/0/1, 2/0/1/1, 3/0/2/1, 4/0/3/1, 5/0/3/2, 6/2/4/2}
{\arraycode}
\end{tikzpicture}
&
\begin{tikzpicture}[baseline=\bl cm]
\foreach \i/\j/\k/\l in {0/0/0/1, 1/1/1/0, 2/1/1/1, 3/2/2/1, 4/3/3/1, 5/3/3/2, 6/2/4/2}
{\arraycode}
\end{tikzpicture}
&
\begin{tikzpicture}[baseline=\bl cm]
\foreach \i/\j/\k/\l in {0/0/0/1, 1/1/1/0, 2/1/1/1, 3/0/2/2, 4/1/1/2, 5/1/1/3, 6/2/2/3}
{\arraycode}
\end{tikzpicture}
&
\begin{tikzpicture}[baseline=\bl cm]
\foreach \i/\j/\k/\l in {0/1/0/0, 1/0/1/1, 2/0/0/2, 3/1/1/1, 4/1/2/2, 5/1/3/3, 6/2/3/3}
{\arraycode}
\end{tikzpicture}\phantom{a}

\\[0.7cm]\hline & & & & &\\[-2mm]
\tikz[baseline=-1mm]\node[rotate=0]{\(\rho=\overline 1\)};& \begin{array}{c}
\alpha_1 \\[-1mm] \alpha_2 \\[-1mm] \alpha_3
\end{array} &\hspace{-0.3cm}
\begin{tikzpicture}[baseline=\bl cm]
\foreach \i/\j/\k/\l in {0/1/0/0, 1/1/1/0, 2/1/1/1, 3/1/2/1, 4/1/3/1, 5/1/3/2, 6/1/4/2}
{\arraycode}
\end{tikzpicture}
&
\begin{tikzpicture}[baseline=\bl cm]
\foreach \i/\j/\k/\l in {0/1/0/0, 1/0/1/0, 2/0/1/1, 3/1/2/1, 4/2/3/1, 5/2/3/2, 6/3/4/2}
{\arraycode}
\end{tikzpicture}
&
\begin{tikzpicture}[baseline=\bl cm]
\foreach \i/\j/\k/\l in {0/1/0/0, 1/0/1/0, 2/1/0/1, 3/0/1/1, 4/0/1/2, 5/1/2/2, 6/1/2/3}
{\arraycode}
\end{tikzpicture}
&
\begin{tikzpicture}[baseline=\bl cm]
\foreach \i/\j/\k/\l in {0/0/1/0, 1/0/0/1, 2/1/1/0, 3/0/1/2, 4/1/1/2, 5/1/2/1, 6/1/2/3}
{\arraycode}
\end{tikzpicture}\phantom{a}\\[0.7cm]\hline
\end{array}\)
\end{center}
\caption{Negative roots for each root system \(\Pi_\Xi\) in \(G(3)\).}
\label{negative roots G3}
\end{table}

The symbols \tikz[baseline=-1mm]\draw[densely dotted, fill=\czero] circle (\a);,
\tikz[baseline=-1mm]\draw[fill=\cone] circle (\a);,
\tikz[baseline=-1mm]\foreach\i in {0,1}\draw[fill=\ctwo] circle (\a-\i*\b);,
\tikz[baseline=-1mm]\foreach\i in {0,1,2}\draw[fill=\cthree] circle (\a-\i*\b);,
\tikz[baseline=-1mm]\foreach\i in {0,1,2,3}\draw[fill=\cfour] circle (\a-\i*\b); in Table \ref{negative roots G3} represent numbers \(0,-1,-2,-3,-4\) respectively, and are displayed in \(3\times 7\) arrays, where each column corresponds to a different root \(\alpha= m_1\alpha_1 + m_2\alpha_2 + m_3\alpha_3\), and number in row \(i\) gives the coefficient \(m_i\) of that root.

Observe that for each root system we have 7 even and 7 odd positive roots. Those numbers double when we consider the corresponding negative roots. The even part is increased by $\dim\mathfrak{h}=3$ because of
the Cartan subalgebra. This gives the expected total of \(\dim(G(3))=(17|14)\).

Given $\Xi,\chh$ we deduce from Table \ref{negative roots G3} the grading \eqref{grading} of $\m^\Xi_\cch\subset\g$,
where $\g_{-s}$ is the sum of root vectors $e_\alpha$, $\alpha= m_1\alpha_1 + m_2\alpha_2 + m_3\alpha_3$,
and $-s=m_1+m_2+m_3$ is the weight of the corresponding column. The positive grading of $\g$ is defined
simlarly, and $\g_0$ is the sum of the root vectors with zero grading and the Cartan subalgebra $\mathfrak{h}$.
Then $\m^\Xi_\cch=\sum_{k<0}\g_k$ and $\fp^\Xi_\cch=\sum_{k\ge0}\g_k$.

Both superspaces $M^\Xi_\cch=\exp\bigl(\m^\Xi_\cch\bigr)$ and $G/P^\Xi_\cch$ contain a vector superdistribution $\D$
corresponding to $\g_{-1}$.
Since $[e_\alpha,e_\beta]=n_{\alpha\beta}e_{\alpha+\beta}$ for some $n_{\alpha\beta}\neq0$
whenever $\alpha+\beta\in\Delta$ is a 
root, the structure of the Carnot algebra
can be derived from Table \ref{negative roots G3} (and supplementary maple file).

Thus the supergeometries associated to $G(3)$ are vector super-distributions, except for two
cases marked as dark nodes on Figure \ref{geometries_diagram}. Those are two contact gradings
(odd and mixed) and the Tanaka-Weisfeiler prolongation of $\m$ in those cases is infinite (contact superalgebras).
The structure group reduction $G_0\subset COSp(\g_{-1})$ , as indicated in Table \ref{reduction G3},
is required in these exceptional cases.

\begin{table}[H]\begin{center}
\(
\begin{array}{|c|c|c|c|}
\hline
\text{Symbol algebra } \mathfrak{m} & \text{Prolongation} \op{pr}(\m) &
\text{Algebra } \op{pr}_0(\m)=\mathfrak{der}_0(\m) & \text{Structure reduction } \mathfrak{g}_0\\
\hline
\mathfrak{m}^\text{I}_1 & \mathfrak{k}(1|7) & \mathfrak{co}(7) & G_2\oplus\C \\
\mathfrak{m}^\text{III}_1 = \mathfrak{m}^\text{IV}_1 & \mathfrak{k}(5|4) & \mathfrak{cspo}(4|4) &\mathfrak{cosp}(3|2)\\
\hline
\end{array}
\)
\caption{Exceptional \(G(3)\)-supergeometries
with \(\dim\op{pr}(\mathfrak{m}^{\Xi}_{\cch})=\infty\).}
\label{reduction G3}
\end{center}
\end{table}

\section{\(F(4)\)-supergeometries}\label{S4}

For \(F(4)\) we have 6 different root systems up to \(\mathcal{W}\)-equivalence \cite{FSS},
shown in Figure \ref{F4dynkin}. Again, the dashed nodes mark odd reflections.
Both even and odd reflections form Weyl groupoid \cite{Se} which acts transitively on the conjugacy classes
of simple root systems.

\begin{figure}[H]
\begin{center}
\includegraphics[width=0.7\textwidth]{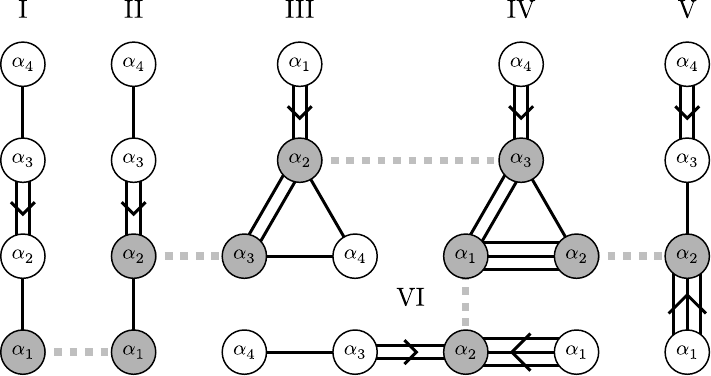}
\end{center}
\caption{\(\mathcal{W}\)-inequivalent Dynkin diagrams of \(F(4)\).}
\label{F4dynkin}
\end{figure}

The \(F(4)\) root system \(\Delta = \Delta_{\overline 0} \cup \Delta_{\overline 1} \subset \mathfrak{h}^*\setminus \{0\}\) is given by
\begin{equation}
\Delta _{\overline 0}= \left\{ \pm \delta, \pm \varepsilon_i, \pm(\varepsilon_i\pm\varepsilon_j) \right\} \,,\quad\quad \Delta_{\overline 1} = \left\{ (\pm \delta \pm \varepsilon_1 \pm \varepsilon_2 \pm \varepsilon_3)/2 \right\}\,,
\label{even_odd_dual_cartan-F4}
\end{equation}
where \(\delta, \varepsilon_1,\varepsilon_2,\varepsilon_3\) are vectors in \(\mathfrak{h}^*\) satisfying \(\langle \varepsilon_i,\varepsilon_j \rangle =\delta _{ij}\), \(\langle \delta,\delta \rangle =-3\) and \(\langle \varepsilon_i,\delta \rangle =0\) with respect to the Killing form. Through these vectors, we define the 6 simple root systems \(\Pi_\Xi = \left\{ \alpha_1, \alpha_2, \alpha_3, \alpha_4 \right\} \) shown in Table \ref{roots F4}, see \cite{ST} for details.

\begin{table}[H]
\begin{center}
\(\begin{array}{|c|ccc|}
\hline
& \text{I} & \text{II} & \text{III}\\\hline
\alpha_1 & (\delta-\varepsilon_1-\varepsilon_2-\varepsilon_3)/2 & (-\delta+\varepsilon_1+\varepsilon_2+\varepsilon_3)/2 & \varepsilon_1-\varepsilon_2\\
\alpha_2 & \varepsilon_3 & (\delta-\varepsilon_1-\varepsilon_2+\varepsilon_3)/2 & (\delta-\varepsilon_1+\varepsilon_2-\varepsilon_3)/2\\
\alpha_3 & \varepsilon_2-\varepsilon_3 & \varepsilon_2-\varepsilon_3 & (-\delta+\varepsilon_1+\varepsilon_2-\varepsilon_3)/2\\
\alpha_4 & \varepsilon_1-\varepsilon_2 & \varepsilon_1-\varepsilon_2 & \varepsilon_3
\\\hline\hline
& \text{IV} & \text{V} & \text{VI}\\\hline
\alpha_1 & (\delta+\varepsilon_1-\varepsilon_2-\varepsilon_3)/2 & \delta & \delta\\
\alpha_2 & (\delta-\varepsilon_1+\varepsilon_2+\varepsilon_3)/2 & (-\delta+\varepsilon_1-\varepsilon_2-\varepsilon_3)/2 & (-\delta-\varepsilon_1+\varepsilon_2+\varepsilon_3)/2\\
\alpha_3 & (-\delta+\varepsilon_1-\varepsilon_2+\varepsilon_3)/2 & \varepsilon_3 & \varepsilon_1-\varepsilon_2\\
\alpha_4 & \varepsilon_2-\varepsilon_3 &  \varepsilon_2-\varepsilon_3 &  \varepsilon_2-\varepsilon_3
\\\hline
\end{array}\)
\end{center}
\caption{\(\Pi_\Xi\) for each \(\Xi\in \left\{ \text{I}, \text{II}, \text{III}, \text{IV}, \text{V}, \text{VI} \right\} \) in \(F(4)\).}
\label{roots F4}
\end{table}

Analogously as in \(G(3)\), we obtain the map of \(F(4)\)-supergeometries, shown in Figure \ref{F4map}, and negative roots for each simple root system \(\Pi_\Xi\) in table \ref{negative roots F4}.

\begin{figure}[H]
\begin{center}
\includegraphics[width=\textwidth]{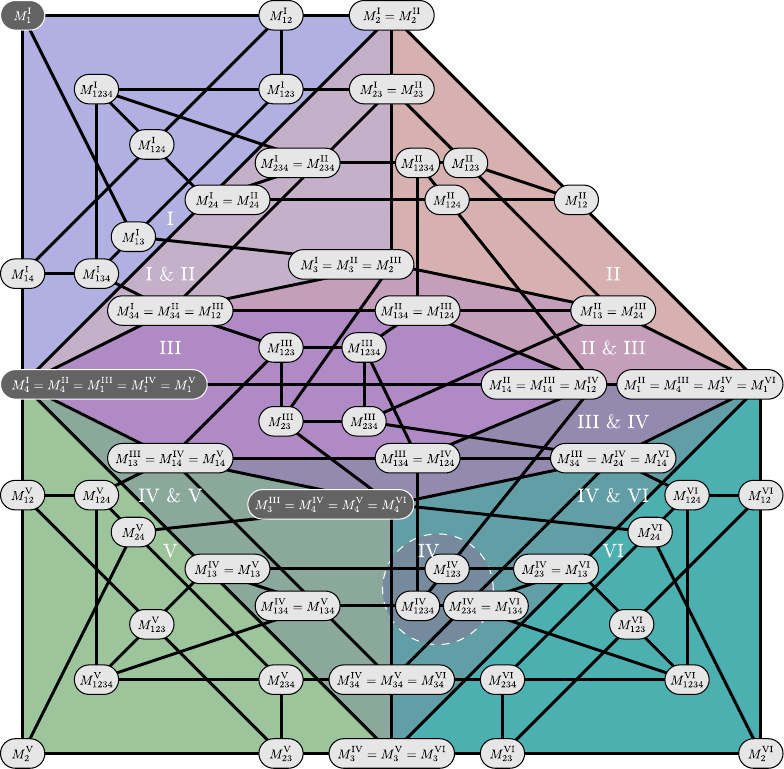}
\end{center}
\caption{Map of \(F(4)\)-supergeometries. Dark nodes correspond to supergeometries that require a \(\mathfrak{g}_0\)-reduction. Edges in the diagram connect supergeometries \(M ^{\Xi}_{\cch_1}\) and \(M ^{\Xi}_{\cch_2}\) such that \(|\chh_1\ominus\chh_2|=1\).}
\label{F4map}
\end{figure}

As opposed to Figure \ref{geometries_diagram}, in Figure \ref{F4map} the domains \(\Xi\) overlap (in this case
such diagram should be rather 3-dimensional to avoid overlaps).
In Figure \ref{F4map} the domains I, V, VI have the same shape. Domains II and IV also have the same shape among them, while domain III uses the central region of Figure \ref{F4map}.

Observe in Table \ref{negative roots F4} that we have 10 even roots and 8 odd negative roots. These numbers double when we consider the positive roots counterparts, and the even part is increased by $\dim\mathfrak{h}=4$ because of
the Cartan subalgebra. This gives the expected total of \(\operatorname{dim}(F(4))=(24|16)\).

Very similar to $G(3)$ case, the superspaces $M^\Xi_\cch=\exp\bigl(\m^\Xi_\cch\bigr)$ or $G/P^\Xi_\cch$
in the $F(4)$ case carry vector distributions $\D$, for which the associated to their derived flag filtration is the grading
of $\g$ corresponding to parabolic $\fp^\Xi_\cch$.

\pgfmathsetmacro{\bl}{3*\h+\a}
\def\arraycode{
\numbertocolor{\j}
\ifthenelse{\j>0}{
\foreach \m in {1,...,\j}
\draw[fill=\c] (\i*\s,1*\h) circle ({\a-(\m-1)*\b});}{
\draw[densely dotted, fill=\c] (\i*\s,1*\h) circle (\a);}
\numbertocolor{\k}
\ifthenelse{\k>0}{
\foreach \m in {1,...,\k}
\draw[fill=\c] (\i*\s,2*\h) circle ({\a-(\m-1)*\b});}{
\draw[densely dotted, fill=\c] (\i*\s,2*\h) circle (\a);}
\numbertocolor{\l}
\ifthenelse{\l>0}{
\foreach \m in {1,...,\l}
\draw[fill=\c] (\i*\s,3*\h) circle ({\a-(\m-1)*\b});}{
\draw[densely dotted, fill=\c] (\i*\s,3*\h) circle (\a);}
\numbertocolor{\p}
\ifthenelse{\p>0}{
\foreach \m in {1,...,\p}
\draw[fill=\c] (\i*\s,4*\h) circle ({\a-(\m-1)*\b});}{
\draw[densely dotted, fill=\c] (\i*\s,4*\h) circle (\a);}}

\begin{table}[H]
\begin{center}
\(\begin{array}{|c|cccc|}\hline
\tikz[baseline=-1mm]\node{\(\Delta^{\Xi}_{\rho,-}\)};& \Xi: &\text{I}\hspace{0.2cm}\phantom{i} & \text{II} & \text{III}\phantom{a} \\\hline & & & &\\[-2mm]
\tikz[baseline=-1mm]\node[rotate=0]{\(\rho=\overline 0\)};& \begin{array}{c}
\alpha_1 \\[-1mm] \alpha_2 \\[-1mm] \alpha_3 \\[-1mm] \alpha_4
\end{array} &\hspace{-0.3cm}
\begin{tikzpicture}[baseline=\bl cm]
\foreach \i/\j/\k/\l/\p in {0/0/1/0/0, 1/0/0/1/0, 2/0/0/0/1, 3/0/1/1/0, 4/0/0/1/1, 5/0/1/1/1, 6/0/2/1/0, 7/0/2/1/1, 8/0/2/2/1, 9/2/3/2/1}
{\arraycode}
\end{tikzpicture}
&
\begin{tikzpicture}[baseline=\bl cm]
\foreach \i/\j/\k/\l/\p in {0/0/0/1/0, 1/0/0/0/1, 2/1/1/0/0, 3/0/0/1/1, 4/1/1/1/0, 5/1/1/1/1, 6/2/2/1/0, 7/2/2/1/1, 8/2/2/2/1, 9/1/3/2/1}
{\arraycode}
\end{tikzpicture}
&
\begin{tikzpicture}[baseline=\bl cm]
\foreach \i/\j/\k/\l/\p in {0/1/0/0/0, 1/0/0/0/1, 2/0/1/1/0, 3/1/1/1/0, 4/0/1/1/1, 5/1/1/1/1, 6/1/2/0/1, 7/0/1/1/2, 8/1/1/1/2, 9/1/2/2/2}
{\arraycode}
\end{tikzpicture}

\\[0.7cm]\hline & & & &\\[-2mm]
\tikz[baseline=-1mm]\node[rotate=0]{\(\rho=\overline 1\)};& \begin{array}{c}
\alpha_1 \\[-1mm] \alpha_2 \\[-1mm] \alpha_3 \\[-1mm] \alpha_4
\end{array} &\hspace{-3mm}
\begin{tikzpicture}[baseline=\bl cm]
\foreach \i/\j/\k/\l/\p in {0/1/0/0/0, 1/1/1/0/0, 2/1/1/1/0, 3/1/1/1/1, 4/1/2/1/0, 5/1/2/1/1, 6/1/2/2/1, 7/1/3/2/1}
{\arraycode}
\end{tikzpicture}
&
\begin{tikzpicture}[baseline=\bl cm]
\foreach \i/\j/\k/\l/\p in {0/1/0/0/0, 1/0/1/0/0, 2/0/1/1/0, 3/0/1/1/1, 4/1/2/1/0, 5/1/2/1/1, 6/1/2/2/1, 7/2/3/2/1}
{\arraycode}
\end{tikzpicture}
&
\begin{tikzpicture}[baseline=\bl cm]
\foreach \i/\j/\k/\l/\p in {0/0/1/0/0, 1/0/0/1/0, 2/1/1/0/0, 3/0/1/0/1, 4/0/0/1/1, 5/1/1/0/1, 6/1/2/1/1, 7/1/2/1/2}
{\arraycode}
\end{tikzpicture}\\[0.7cm]\hline\hline
\tikz[baseline=-1mm]\node{\(\Delta^{\Xi}_{\rho,-}\)};& \Xi: &\text{IV}\hspace{0.2cm}\phantom{i} & \text{V} & \text{VI}\phantom{a} \\\hline & & & &\\[-2mm]
\tikz[baseline=-1mm]\node[rotate=0]{\(\rho=\overline 0\)};& \begin{array}{c}
\alpha_1 \\[-1mm] \alpha_2 \\[-1mm] \alpha_3 \\[-1mm] \alpha_4
\end{array} &\hspace{-0.3cm}
\begin{tikzpicture}[baseline=\bl cm]
\foreach \i/\j/\k/\l/\p in {0/0/0/0/1, 1/1/1/0/0, 2/1/0/1/0, 3/0/1/1/0, 4/1/0/1/1, 5/0/1/1/1, 6/0/2/2/1, 7/1/1/2/1, 8/1/2/3/1, 9/1/2/3/2}
{\arraycode}
\end{tikzpicture}
&
\begin{tikzpicture}[baseline=\bl cm]
\foreach \i/\j/\k/\l/\p in {0/1/0/0/0, 1/0/0/1/0, 2/0/0/0/1, 3/0/0/1/1, 4/0/0/2/1, 5/1/2/1/0, 6/1/2/1/1, 7/1/2/2/1, 8/1/2/3/1, 9/1/2/3/2}
{\arraycode}
\end{tikzpicture}
&
\begin{tikzpicture}[baseline=\bl cm]
\foreach \i/\j/\k/\l/\p in {0/1/0/0/0, 1/0/0/1/0, 2/0/0/0/1, 3/0/0/1/1, 4/1/2/1/0, 5/1/2/1/1, 6/1/2/2/1, 7/2/4/2/1, 8/2/4/3/1, 9/2/4/3/2}
{\arraycode}
\end{tikzpicture}

\\[0.7cm]\hline & & & &\\[-2mm]
\tikz[baseline=-1mm]\node[rotate=0]{\(\rho=\overline 1\)};& \begin{array}{c}
\alpha_1 \\[-1mm] \alpha_2 \\[-1mm] \alpha_3 \\[-1mm] \alpha_4
\end{array} &\hspace{-3mm}
\begin{tikzpicture}[baseline=\bl cm]
\foreach \i/\j/\k/\l/\p in {0/1/0/0/0, 1/0/1/0/0, 2/0/0/1/0, 3/0/0/1/1, 4/1/1/1/0, 5/1/1/1/1, 6/0/1/2/1, 7/1/2/2/1}
{\arraycode}
\end{tikzpicture}
&
\begin{tikzpicture}[baseline=\bl cm]
\foreach \i/\j/\k/\l/\p in {0/0/1/0/0, 1/1/1/0/0, 2/0/1/1/0, 3/1/1/1/0, 4/0/1/1/1, 5/1/1/1/1, 6/0/1/2/1, 7/1/1/2/1}
{\arraycode}
\end{tikzpicture}
&
\begin{tikzpicture}[baseline=\bl cm]
\foreach \i/\j/\k/\l/\p in {0/0/1/0/0, 1/1/1/0/0, 2/0/1/1/0, 3/1/1/1/0, 4/0/1/1/1, 5/1/1/1/1, 6/1/3/2/1, 7/2/3/2/1}
{\arraycode}
\end{tikzpicture}\\[0.7cm]\hline\end{array}\)
\end{center}
\caption{Negative roots for each root system \(\Pi_\Xi\) in \(F(4)\).}
\label{negative roots F4}
\end{table}

Thus the supergeometries associated to $F(4)$ are vector super-distributions, except for three
cases marked as dark nodes on Figure \ref{F4map}. Those are one irreducible and two contact 
(odd and mixed) gradings, and the Tanaka-Weisfeiler prolongation of $\m$ in such cases is infinite
(all vector fields and contact vector fields, respectively). The structure group reductions $G_0\subset GL(\m)$ or
$G_0\subset COSp(\g_{-1})$, as indicated in Table \ref{reduction F4}, is required in these exceptional cases.

\begin{table}[H]\begin{center}
\(
\begin{array}{|c|c|c|c|}
\hline
\text{Symbol algebra } \mathfrak{m} & \text{Prolongation} \op{pr}(\m) &
\text{Algebra } \op{pr}_0(\m)=\mathfrak{der}_0(\m) & \text{Structure reduction } \mathfrak{g}_0\\
\hline
\mathfrak{m}^\text{I}_1 & \mathfrak{k}(1|8) & \mathfrak{co}(8) & \mathfrak{cspin}(7)\\
\mathfrak{m}^\text{I}_4 = \mathfrak{m}^\text{II}_4 = \mathfrak{m}^\text{III}_1 = \mathfrak{m}^\text{IV}_1 = \mathfrak{m}^\text{V}_1 &
\mathfrak{vect}(6|4) & \mathfrak{gl}(6|4) & \mathfrak{cosp}(2|4)\\
\mathfrak{m}^\text{III}_3 = \mathfrak{m}^\text{IV}_4 = \mathfrak{m}^\text{V}_4 = \mathfrak{m}^\text{VI}_4 &
\mathfrak{k}(7|4) & \mathfrak{cosp}(6|4) & \mathfrak{cosp}(4|2;\tfrac{1}{2})\\
\hline
\end{array}
\)
\caption{Exceptional \(F(4)\)-supergeometries
with \(\dim\op{pr}(\mathfrak{m}^{\Xi}_{\cch})=\infty\).}
\label{reduction F4}
\end{center}
\end{table}

\section{Computing the prolongation and proof of the main result}\label{S5}

\subsection{Encoding $\m$}

For each choice \(\Xi\) of simple root system/Dynkin diagram we have even and odd roots, described in Tables \ref{negative roots G3} and \ref{negative roots F4}.
Every negative root is of the form \(\alpha = \sum m_i\alpha_i\) with all \(m_i\leq 0\) and it inherits the 
associated grading from \(P^\Xi_{\cch}\) given by \(\sum_{k\in\cch} m_k \in \left\{ -1,\dots,-\mu \right\} \).

The goal is to reconstruct the Lie superalgebra \(\mathfrak{m}^{\Xi}_{\cch} = \mathfrak{g}^{\Xi,\cch}_{-\mu}\oplus\dots\oplus \mathfrak{g}^{\Xi,\cch}_{-1}\), together with a possibly requested 
\(\mathfrak{g}^{\Xi,\cch}_0\)-reduction (we will see that we don't need higher order reductions) such that
\[\mathfrak{g}^{\Xi,\cch}=\op{pr}(\mathfrak{m}^{\Xi}_{\cch}, \mathfrak{g}^{\Xi,\cch}_{0})= \mathfrak{g}^{\Xi,\cch}_{-\mu}\oplus\dots\oplus \mathfrak{g}^{\Xi,\cch}_{\mu}\]
is the Lie superalgebra of the supergeometry \(M^{\Xi}_{\cch}\).

For every \(\alpha_i, \alpha_j\in \Delta^\Xi_-\), if
\begin{equation}
\alpha_i+\alpha_j=\alpha_k\in\Delta^\Xi_-
\label{sum of roots}
\end{equation}
for some \(k\), then the Lie bracket of our supergeometry satisfies \([\, v_i\,,\, v_j\,]=c_{ij}^k v_k\) (no summation) for some unknown nonzero structure constant \(c_{ij}^k\), where \(v_\ell\) is an associated root vector to root \(\alpha_\ell\) (scale is not fixed yet). Our approach is to find the structure constants for the maximal parabolics (\(\chh\) being maximal, for each \(\Xi\)), keeping the root vectors multigraded by $\chh$. Then, a change of \(\chh\) to a subset is translated 
into a change in grading of the roots vectors, while the structure relations themselves remain unchanged.

For the maximal parabolic (Borel) cases, by rescaling the root vectors, fix as many of the structure constants 
\(c_{ij}^k=1\) as possible for the indices \(k\) occurring in relations \eqref{sum of roots}. 
The remaining structure constants \(c_{ij}^k\) are computed through the Jacobi identity.
The normalized root vectors together with the obtained structure relations form a Lie superalgebra \(\mathfrak{m}^{\Xi}_{\cch} = \mathfrak{g}^{\Xi,\cch}_{-\mu}\oplus\dots\oplus \mathfrak{g}^{\Xi,\cch}_{-1}\)
(a simplified version of Serre's relations).

\subsection{The algorithm}\label{algorithm}

We compute $\mathfrak{g}_i\subset\op{pr}(\m)$ for \(i\geq0\) iteratively as follows. Consider the \(\mathbb Z\)-graded symbol algebra 
$\m=\oplus _{i>0}^\mu\mathfrak{g}_{-i}$ and assume that the algorithm has already produced 
 \[
\op{pr}_{<k}(\m)=\g_{-\mu}\oplus\dots\oplus\g_{-1}\oplus\dots\oplus\g_{k-1}.
 \]
In practical terms, this means that we got basis elements for \(\mathfrak{g}_i\) of pure grading and parity, 
and all Lie brackets that involve only elements of \(\mathfrak{g}_{<k}\).

The next step is to generate \(\mathfrak{g}_k\) through linear maps of the form \(A: \mathfrak{g}_{-1}\longrightarrow \mathfrak{g}_{k-1}\) between vector spaces, which have dimensions \(\operatorname{dim}(\mathfrak{g}_{-1})=(n_{\overline 0}\,|\,n_{\overline 1})\) and \(\operatorname{dim}(\mathfrak{g}_{k-1})=(m_{\overline 0}\,|\,m_{\overline 1})\), and hence \(A=(a_{ij})\) is a \((m_{\overline 0}+m_{\overline 1})\times (n_{\overline 0}+n_{\overline 1})\) matrix. \(A\) is considered to be
 \[
\left\{ \begin{array}{cl}
\text{even} & \text{if } a_{ij}=0 \text{ whenever } ((i\leq m_{\overline 0}) \wedge (j>n_{\overline 0}))\vee ((i>m_{\overline 0})\wedge(j\leq n_{\overline 0}))\,,\\
\text{odd} &  \text{if } a_{ij}=0 \text{ whenever } ((i\leq m_{\overline 0})\wedge(j\leq n_{\overline 0}))\vee ((i>m_{\overline 0})\wedge( j> n_{\overline 0}))\,.
\end{array} \right.
 \]
The zero matrix is both even and odd, and a nontrivial combination of even and odd matrices is neither even nor odd. 
\(A\) can always be decomposed into the sum \(A_{\overline 0}+A_{\overline 1}\) of even and odd elements.
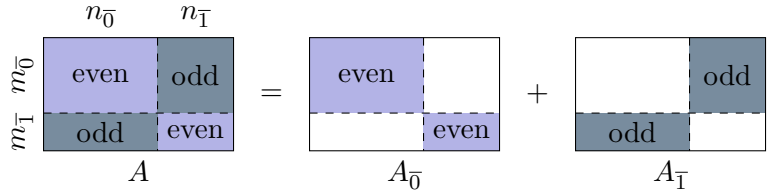
\begin{figure}[H]
\begin{center}
\begin{tikzpicture}
\def\no{1.5}
\def\nl{1}
\def\mo{1}
\def\ml{0.5}
\def\sep{1}
\def\sepl{0}
\definecolor{color even}{rgb}{0.7,0.7,0.9}
\definecolor{color odd}{rgb}{0.47, 0.55, 0.62}
\begin{scope}[fill=color even]
\fill (0,0) rectangle (\no, -\mo);
\fill (\no,-\mo) rectangle (\no+\nl, -\mo-\ml);
\fill (\no+\nl+\sep,0) rectangle (\no+\nl+\sep+\no, -\mo);
\fill (\no+\nl+\sep+\no,-\mo) rectangle (\no+\nl+\sep+\no+\nl, -\mo-\ml);
\end{scope}
\begin{scope}[fill=color odd]
\fill (\no,0) rectangle (\no+\nl, -\mo);
\fill (0,-\mo) rectangle (\no, -\mo-\ml);
\fill (2*\no+2*\nl+2*\sep+\no,0) rectangle (2*\no+2*\nl+2*\sep+\no+\nl, -\mo);
\fill (2*\no+2*\nl+2*\sep,-\mo) rectangle (2*\no+2*\nl+2*\sep+\no, -\mo-\ml);
\end{scope}
\draw[dashed] (\no, 0) to (\no, -\mo-\ml);
\draw[dashed] (0, -\mo) to (\no+\nl, -\mo);
\draw[dashed] (\no+\nl+\sep+\no, 0) to (\no+\nl+\sep+\no, -\mo-\ml);
\draw[dashed] (\no+\nl+\sep+0, -\mo) to (\no+\nl+\sep+\no+\nl, -\mo);
\draw[dashed] (\no+\nl+\sep+\no+\nl+\sep+\no, 0) to (\no+\nl+\sep+\no+\nl+\sep+\no, -\mo-\ml);
\draw[dashed] (\no+\nl+\sep+\no+\nl+\sep+0, -\mo) to (\no+\nl+\sep+\no+\nl+\sep+\no+\nl, -\mo);
\draw (0,0) rectangle (\no+\nl, -\mo-\ml);
\draw (\no+\nl+\sep,0) rectangle (\no+\nl+\sep+\no+\nl, -\mo-\ml);
\draw (\no+\nl+\sep+\no+\nl+\sep,0) rectangle (\no+\nl+\sep+\no+\nl+\sep+\no+\nl, -\mo-\ml);
\node at (\no+\nl+\sep*0.5, -\mo*0.5-\ml*0.5) {\(=\)};
\node at (\no+\nl+\sep+\no+\nl+\sep*0.5, -\mo*0.5-\ml*0.5) {\(+\)};
\node at (\no*0.5, -\mo*0.5){even};
\node at (\no+\nl*0.5, -\mo-\ml*0.5){even};
\node at (\no+\nl+\sep+\no*0.5, -\mo*0.5){even};
\node at (\no+\nl+\sep+\no+\nl*0.5, -\mo-\ml*0.5){even};
\node at (\no+\nl*0.5, -\mo*0.5){odd};
\node at (\no*0.5, -\mo-\ml*0.5){odd};
\node at (\no+\nl+\sep+\no+\nl+\sep+\no+\nl*0.5, -\mo*0.5){odd};
\node at (\no+\nl+\sep+\no+\nl+\sep+\no*0.5, -\mo-\ml*0.5){odd};
\node[above] at (\no/2, \sepl){\(n_{\overline 0}\)};
\node[above] at (\no+\nl/2, \sepl){\(n_{\overline 1}\)};
\node[rotate=90, above] at (-\sepl, -\mo/2){\(m_{\overline 0}\)};
\node[rotate=90, above] at (-\sepl, -\mo-\ml/2){\(m_{\overline 1}\)};
\node[below] at (\no/2+\nl/2, -\mo-\ml-\sepl){\(A\)};
\node[below] at (\no+\nl+\sep+\no/2+\nl/2, -\mo-\ml-\sepl){\(A _{\overline 0}\)};
\node[below] at (\no+\nl+\sep+\no+\nl+\sep+\no/2+\nl/2, -\mo-\ml-\sepl){\(A _{\overline 1}\)};
\end{tikzpicture}
\end{center}
\caption{Decomposition of matrix \(A\) into even and odd parts.}
\label{decomposition matrix}
\end{figure}

Since \(\mathcal{D}\) is bracket-generating, then \(\mathfrak{m}\) is fundamental, so the basis elements of \(\mathfrak{m}\) are of the recursive form
\[
v_j:(v_j \in \mathfrak{g}_{-1})\vee (v_j\propto [\, w_j\,,\, v_i,] \text{ for some } w_j\in \mathfrak{g}_{-1} \text{ and }
v_i,i<j, \text{ constructed earlier}).
\]
This recursion gives a recipe for extending \(A\) to linear maps \(\g_t\longrightarrow \g_{k+t}\) 
for \(t\in \left\{ -2,\dots,-\mu \right\} \), via the Leibniz super-rule
\begin{equation}
A([\, v\,,\, w\,])= [\, A(v)\,,\, w\,] + [\, v\,,\, A_{\overline 0}(w)] + (-1)^{|v|}[\, v\,,\, A_{\overline 1} (w)\,]
\label{leibniz}
\end{equation}
for pure parity \(v,w\in\m\). 
An application of $A$ to any relation $\sum r_s[v_s,w_s]=0$ via \eqref{leibniz} gives linear constraints
on the coefficients \(a_{ij}\). Let \(\ell\) be the dimension of the solution space of these constraints.
This means that \(A\) is a matrix with coefficients linearly expressed by a maximal set of independent parameters
\(\left\{ a_1,\dots,a_\ell \right\}\), which we can assume to be of pure parity. 
Decompose the matrix by them: $A=\sum_{i=1}^\ell a_iA_i$, where
 \[
A_i=A|_{\{a_j=\delta_{ij}\} }
 \]
induces a basis \(e_i\) of \(\mathfrak{g}_k\) (\(\delta_{ij}\) is the Kronecker delta), with $\ell _{\overline 0}$ even
elements and $\ell _{\overline 1}$ odd elements, \(\ell = \ell _{\overline 0} + \ell _{\overline 1}\).
The basis element \(e_i\) has a well defined parity, and the structure constants in 
\([\,e_i\,,\,v_j\,]=\sum_s c_{ij}^s v_s\) are given by \(c_{ij}^s = (A_i(v_j))^s\), where \(( v_1,\dots,v_{n})\) is the chosen basis of \(\m\). This gives the brackets of $\g_k$ and $\m$.

To compute the brackets of $\g_{\ge0}$ landing in $\g_k$ choose
\(i\in \left\{ 0,\dots, \left\lfloor k/2\right\rfloor \right\} \), \(u\in \g_{k-i}\), \( w\in\g_{i}\).
Then for each basis element \(v\in\g_{-1}\) we get
 \[
[\, [\, u\,,\, w\,]\,,\, v\,] = [\, u\,,\, [\, w\,,\, v\,]\,]-(-1)^{|u||w|}[\, w\,,\, [\, u\,,\, v\,]\,]\in \mathfrak{g}_{k-1}\,.
 \]
Decomposing the result by a basis of \(\g_{k-1}\) and arranging the coefficients in columns we get a matrix $A$
(columns correspond to a basis \(\{v\in\g_{-1}\}\)). Decomposing in turn $A$ by $A_i$,
corresponding to the basis $e_i$ of $\g_k$, we get the required structure constants. 

The dimension of the newly constructed \(\mathfrak{g}_k\) is \((\ell _{\overline 0}\,|\,\ell _{\overline 1})\).
If \(\ell=0\), the algorithm stops. Otherwise, the algorithm proceeds to the level \(k+1\). If the algorithm stops 
after a finite number of iterations, it produces a Lie superalgebra, which can be verified by the Jacobi identity. 
To avoid infinite loops in the computation a threshold amount of iterations is imposed in the algorithm.

\subsection{The proof}

We compute the Tanaka-Weisfeiler prolongation \(\op{pr}(\m^{\Xi}_{\cch})\) using a specially designed package for \textsc{Maple}. 
In most cases this computation already produces the Lie algebra $\g$, but for the dark-filled nodes in Figures \ref{geometries_diagram} and \ref{F4map}, 
a \(\mathfrak{g}_0\)-reduction is required.

This reduction specifies a Lie subalgebra \(\mathfrak{g}_0\) of the superalgebra \(\op{pr}_0(\m)=\mathfrak{der}_0(m)\) given by the algorithm. 
Once this reduction is given, the algorithm proceeds unchanged but starting from the level \(k=1\) instead of \(k=0\).
In all those dark-filled nodes, the result of this modified prolongation is $\g$.

Details of this computer-based calculation are in the \textsc{Maple} supplement to arXiv version of the current paper. 
Computer operations involve only symbolic manipulations and integer arithmetic, and hence are mathematically rigorous. 
This proves the main theorem, which we reformulate as follows.

 \begin{theorem}
\label{theorem G3}
Let $\g$ be $G(3)$ or $F(4)$ and $\fp^\Xi_\cch\subset\g$ be a choice of parabolic subalgebra with grading of depth \(\mu>0\).
Let \(\m^{\Xi}_{\cch}\) be the nilradical of the opposite parabolic. If \(\m^{\Xi}_{\cch}\) is none of the symbol algebras 
in Tables \ref{reduction G3} or \ref{reduction F4}, then the Tanaka-Weisfeiler prolongation of \(\m^{\Xi}_{\cch}\) is $\g$:
 \[
\op{pr}(\m^{\Xi}_{\cch}) =\g^{\Xi,\cch}_{-\mu}\oplus\dots\oplus\g^{\Xi,\cch}_{-1}\oplus\g^{\Xi,\cch}_0\oplus \g^{\Xi,\cch}_{1}\dots\oplus\g^{\Xi,\cch}_{\mu}=\g
 \]
If \(\m^{\Xi}_{\cch}\) is one of the symbol algebras in Table \ref{reduction G3} or \ref{reduction F4}, then the Tanaka-Weisfeiler prolongation with \(\g^{\Xi,\cch}_{0}\)-reduction is $\g$:
 \[
\op{pr}(\m^{\Xi}_{\cch},\g^{\Xi,\cch}_0) =\g^{\Xi,\cch}_{-\mu}\oplus\dots\oplus\g^{\Xi,\cch}_{-1}\oplus\g^{\Xi,\cch}_0\oplus\g^{\Xi,\cch}_{1}\dots\oplus \g^{\Xi,\cch}_{\mu}=\g.
 \]
 \end{theorem}

This gives both local and global symmetry of supergeometry \(M^{\Xi}_{\cch}\), in a spirit of \cite{O,KST2}.
Moreover our computation implies the following claim for general vector superdistributions with fixed Carnot algebra.

 \begin{cor}\label{cor2}
Let $\D$ be a strongly regular superdistribution on a supermanifold $M$.
Suppose that its symbol is one of $\m^\Xi_\cch$ computed for $\g$ that is $G(3)$ or $F(4)$,
but the geometry is not contact or irreducible.
Then dimension of the symmetry superalgebra of $\D$ does not exceed $\dim\g$, 
that is $(17|14)$ for $G(3)$ and $(24|16)$ for $F(4)$, and symmetry dimension equals to it only if $(M,\D)$ is locally isomorphic 
to our model space $(M^\Xi_\cch,\D^\Xi_\cch)$.
 \end{cor}

\section{Integral curves and surfaces of the superdistributions}\label{S6}


Consider $(1|1)$ integral curves of the invariant distribution $\D$ on one of geometries $M^\Xi_\cch$,
i.e.\ such morphisms $\varphi:\C^{1|1}\to M^\Xi_\cch$ that for any $\omega\in\op{Ann}(\D)$
we have: $\varphi^*\omega=0$. (In this formulation it is clear that we allow degenerations 
to $(1|0)$ and $(0|1)$ integral curves as well as to points.) In particular, such curves arise from
the flow of a mixed vector field $X=X_{\bar0}+X_{\bar1}$ with $[X,X]=0$, defining an Abelian action of $\C^{1|1}$,
cf.\ \cite{GW,MSV}.
Of course, for purely odd distributions $\D$ we can restrict to $(0|1)$ integral curves.

 \begin{theorem}\label{th3}
In 15 of 19 cases of generalized flag varieties for $G(3)$ and in all 55 cases for $F(4)$ the set of
(or the equation for) integral curves 
of the corresponding superdistribution $\D$ on $M^\Xi_\cch$ allows to recover the reduced distribution 
$D\subset TM|_{M_o}$ but not necessary $\D$.
In 4 special cases $M_1^{\rm II}=M_3^{\rm III}=M_3^{\rm IV}$, $M_{13}^{\rm III}=M_{13}^{\rm IV}$,
$M_{23}^{\rm IV}$ and $M_{123}^{\rm IV}$ 
even the reduced distribution $D$ cannot be recovered from $(1|1)$ integral curves.
 \end{theorem}

This statement is important for understanding of encoding superalgebras via differential equations, 
so we provide the proof in several steps, and also discuss related topics.

\subsection{Recovering $D$}
A tangent vector to $(1|1)$ integral curve $\varphi$ at 0 is an element of the pullback tangent bundle  
$\varphi^*\mathcal{T}M$ at 0. For point $x=\varphi(0)\in M_o$ it can be identified with an element in 
$D_x\subset T_xM$, and furthermore can be identified with an element in $\g_{-1}=\g^{-1}\op{mod}\fp$.
This element is, in fact, a pair $(v_{\bar0},v_{\bar1})$ of even and odd vectors or a mixed parity vector
$v=v_{\bar0}+v_{\bar1}$. Because the curve is integral, we have $[v,v]=0$. 

Due to independence of the components
this means $[v_{\bar0},v_{\bar1}]=0$ and $[v_{\bar1},v_{\bar1}]=0$ in $\m\subset\g$
(and trivially $[v_{\bar0},v_{\bar0}]=0$). 
Since degenerations are allowed, we can restrict to pure parity cases when $v_{\bar0}=0$ or $v_{\bar1}=0$.

In particular, since even $(1|0)$ curves have no involutivity restrictions, their tangents span $(\g_{-1})_{\bar0}$.
However odd $(0|1)$ curves have null tangents with respect to the vector-valued quadric $[,]|_{\m_{\bar1}}$,
in other words are subject to the constraints $[v_{\bar1},v_{\bar1}]=0$. 

We claim that in all but 4 special cases the odd tangent vectors span $(\g_{-1})_{\bar1}$. 
To see this for each parabolic $\fp\subset\g$ we consider Dynkin diagrams in turn. 

We start with $G(3)$ and refer to Table \ref{negative roots G3} of negative roots $\alpha$, 
giving a basis $e_\alpha$ of $\m=\g_-$. Since $[e_\alpha,e_\beta]=k_{\alpha,\beta}e_{\alpha+\beta}$
for $k_{\alpha,\beta}\neq0$ iff $\alpha+\beta$ is a root, we see that an odd root vector $e_\alpha$ is null if
$2\alpha$ is not (an even) root. This reads off the diagram graphically 
by watching for columns in even part with even number of circles everywhere, 
and then deciding which parabolics correspond to this root vector being in $\g_{-2}$. 

For the first diagram $\Xi=\text{I}$ the only even double root is $-(2\alpha_1+4\alpha_2+2\alpha_3)$.
The corresponding root vector belongs to $\g_{-2}$ only for parabolics $\fp_1$ and $\fp_3$.
In the first case we get odd contact grading, where $\g_{-1}$ is odd with the bracket being a nondegenerate 
conformal quadric in 7D (this can be also read off the table). In this case the null cone is nondegenerate and spans 
the space $(\g_{-1})_{\bar1}$. In the second case the bracket $\Lambda^2(\g_{-1})_{\bar1}\to(\g_{-2})_{\bar0}$
has 1D range and is a nondegenerate conformal quadric in 3D. Thus again the null cone is nondegenerate and 
spans the space $(\g_{-1})_{\bar1}$.

For the second diagram $\Xi=\text{II}$ the only even double root is $-(2\alpha_1+4\alpha_2+2\alpha_3)$.
The corresponding root vector belongs to $\g_{-2}$ only for parabolics $\fp_1$ and $\fp_3$. The first case is special, 
where the second corresponds to a case considered above because $\fp_3^{\rm II}=\fp_3^{\rm I}$.

For the third diagram $\Xi=\text{III}$ the only even double root is $-(2\alpha_2+2\alpha_3)$.
The corresponding root vector belongs to $\g_{-2}$ for parabolics $\fp_2$, $\fp_{12}$, $\fp_3$ and $\fp_{13}$.
For $\fp_2$  the bracket $\Lambda^2(\g_{-1})_{\bar1}\to(\g_{-2})_{\bar0}$
has 1D range and is a nondegenerate conformal quadric in 3D. Thus the null cone is nondegenerate and 
spans the space $(\g_{-1})_{\bar1}$. 
For $\fp_{12}$ the odd part of the distribution splits 
$(\g_{-1})_{\bar1}=\langle e_{-\alpha_1},e_{-(\alpha_1+\alpha_3)}\rangle\oplus
\langle e_{-\alpha_2},e_{-(\alpha_2+\alpha_3)},e_{-(\alpha_2+2\alpha_3)}\rangle$ 
(into $(\g_0)_{\bar0}$ irreps) and the first summand is null, 
while on the second the bracket has 1D range $\langle e_{-(2\alpha_2+2\alpha_3)}\rangle$
and is a nondegenerate conformal quadric in 3D. Thus again null vectors span the space $(\g_{-1})_{\bar1}$. 
For $\fp_3$ the situation is repeated due to $\fp_3^{\rm III}=\fp_3^{\rm I}$.
For $\fp_{13}$ we have another special case.

For the final diagram $\Xi=\text{IV}$ the only even double root is $-2\alpha_3$.
The corresponding root vector belongs to $\g_{-2}$ only for parabolics $\fp_3$, $\fp_{13}$, $\fp_{23}$ and $\fp_{123}$.
The first two were already considered above. The last two are precisely the other 2 special cases.

Consideration of $F(4)$ is surprisingly simple: it follows from Table \ref{negative roots F4} that,
for any choice of Dynkin diagram, there are no even double roots. 
Thus all odd root vectors in $\g_{-1}$ are null and they span the space $(\g_{-1})_{\bar1}$ for any
choice of parabolic subalgebra $\fp_\cch^\Xi$.

\subsection{Special cases for integral curves}

Let us consider four exceptions for $G(3)$.

For $\fp_1^{\rm II}$ Table \ref{negative roots G3} yields: 
$(\g_{-1})_{\bar1}=\langle e_{-\alpha_1},e_{-(\alpha_1+2\alpha_2+\alpha_3)}\rangle$ 
and $(\g_{-2})_{\bar0}=\langle e_{-(2\alpha_1+2\alpha_2+\alpha_3)},e_{-(2\alpha_1+4\alpha_2+2\alpha_3)}\rangle$.
Thus tangents to $(1|1)$-integral curves span 
$\langle e_{-(\alpha_1+\alpha_2)},e_{-(\alpha_1+\alpha_2+\alpha_3)}|e_{-\alpha_1}\rangle\subsetneq\g_{-1}$.

For $\fp_{13}^{\rm III}$ we have: 
$(\g_{-1})_{\bar1}=\langle e_{-\alpha_1},e_{-(\alpha_2+\alpha_3)}\rangle$ 
and $(\g_{-2})_{\bar0}=\langle e_{-(\alpha_1+\alpha_2+\alpha_3)},e_{-(2\alpha_2+2\alpha_3)}\rangle$.
Thus tangents to $(1|1)$-integral curves span 
$\langle e_{-\alpha_3},e_{-(\alpha_1+\alpha_2)}|e_{-\alpha_1}\rangle\subsetneq\g_{-1}$.

For $\fp_{23}^{\rm IV}$ we have consistent grading: 
$\g_{-1}=\langle e_{-\alpha_2},e_{-\alpha_3},e_{-(\alpha_1+\alpha_2)}\rangle$ ,
 $\g_{-2}=\langle e_{-(\alpha_2+\alpha_3)},e_{-2\alpha_3},e_{-(\alpha_1+\alpha_2+\alpha_3)}\rangle$.
Thus tangents to $(1|1)$-integral curves span 
$\langle e_{-\alpha_2},e_{-(\alpha_1+\alpha_2)}\rangle\subsetneq\g_{-1}$.

For $\fp_{123}^{\rm IV}$ we have: 
$\g_{-1}=\langle e_{-\alpha_1}|e_{-\alpha_2},e_{-\alpha_3}\rangle$ and 
$\g_{-2}=\langle e_{-(\alpha_2+\alpha_3)},e_{-2\alpha_3}|e_{-(\alpha_1+\alpha_2)}\rangle$.
Thus tangents to $(1|1)$-integral curves span 
$\langle e_{-\alpha_1}|e_{-\alpha_2}\rangle\subsetneq\g_{-1}$.

 \begin{remark}
There are some partial fixes for these cases. For instance, in the case of $\fp_{23}^{\rm IV}$ tangents to 
the integral $(1|1)$ curves of the distribution $\D^2$ span $D^2\subset TM|_{M_o}$, which corresponds to 
$\g^{-2}\op{mod}\fp=\g_{-2}\oplus\g_{-1}$. In this space, with the corresponding brackets,
the centralizer of $\langle e_{-\alpha_2},e_{-(\alpha_1+\alpha_2)}\rangle$ is $\g_{-1}$.
 \end{remark}

\subsection{Non-recovering $\D$}

It is important however that integral curves can recover only reduced bundle and not the original distribution $\D$.
Indeed, let $\C^{1|1}$ have coordinates $(x|\theta)$. Then pullbacks of even 1-forms are $a(x)\,dx+d(x)\theta\,d\theta$,
while pullbacks of odd 1-forms are $b(x)\theta\,dx+c(x)\,d\theta$. If coefficients of 1-forms from $\op{Ann}(\D)$
contain nonlinearities in odd coordinates $\xi_i$ on $M$, those will pullback to zero. 
This loss of information cannot be recovered, and hence we cannot recover $\D$ in general.

There are however cases, when generators of $\D$ can have only linear coefficients in odd coordinates on $M$.
This happens when the distribution has depth 2 (and indeed also in the irreducible case $\D=\mathcal{T}M$ of depth 1). 
These cases are: odd contact structure on $M_1^{\rm I}$, mixed contact structure 
on $M_1^{\rm III}$ and depth 2 distribution with growth $(4|3,2|2)$ on $M_3^{\rm I}$ for $G(3)$;
odd contact structure on $\m_1^{\rm I}$, mixed contact structure on $M_3^{\rm III}$, 
irreducible geometry of dimension $(6|4)$ on $M_4^{\rm I}$
and depth 2 distributions with growth $(6|4,2|2)$ on $M_3^{\rm I}$,
growth $(4|4,3|1)$ on $M_1^{\rm II}$,
growth $(0|8,5|0)$ on $M_2^{\rm V}$ for $F(4)$.

In other cases, one cannot avoid quadratic coefficients in $\xi_i$ and hence requires higher dimensional
integral submanifolds to encode the distribution $\D$. For instance, depth 3 distribution with growth 
$(2|4,1|2,2|0)$ on $M_2^{\rm IV}$ considered in \cite{KST} has quadratic coefficients in odd coordinates.
Thus at least 2 odd coordinates are required to encode the geometry. There can be no higher dimensional
integral submanifolds than ones of dimension $(1|2)$, but those are sufficient and they do encode the geometry
of the distribution through the super version of the Hilbert-Cartan equation constructed in \cite{KST}; the solutions 
of that equation are precisely these integral surfaces. This implies 
that the symmetry of the super Hilbert-Cartan is $G(3)$.

\subsection{Differential equations}

One expects to be able to construct, in a similar manner, other realizations of $G(3)$ and $F(4)$ as symmetries of
differential equations, however this will not be possible for every realization $M^\Xi_\cch$. For instance,
integral curves and hence surfaces are insufficient to recover $D$ in 4 special cases from Theorem \ref{th3}, 
so these Lie superalgebras do not realize as supersymmetries on the corresponding manifolds.

This is in contrast with the classical case, when integral curves recover the distribution, and hence the
underdetermined ODE on such curves (so-called Monge equation) has the prescribed symmetry.
For instance, in this way, the Hilbert-Cartan equation $z'(x)=y''(x)^2$ has infinitesimal symmetry algebra $G_2$.
(We refer for historical introduction of \cite{KST} with extended references, in addition to \cite{Ca,En}.)

Let us remark that in the classical case, every realization by vector fields implies realization as symmetries of 
some differential equation. To find such, one prolongs the action to the space of jets and computes differential invariants
of the action. However in supergeometry this may fail.

For the cases, when the geometry on $M=M^\Xi_\cch$ is encoded by a distribution $\D$ 
 (structure reduction will be considered in the next section) 
on generalized flag supervariety $G/P$
 (17 cases for $G(3)$ out of 19 and 52 cases for $F(4)$ out of 55),
the symmetry $\op{Lie}(G)=\g$ is encoded by the Lie equation on vector field $X$:
 $$
L_X\D=\D.
 $$ 
This differential equation can be considered as a submanifold $\mathcal{E}\subset J^1\pi$ 
in the space of 1-jets of sections of the tangent bundle $\pi:TM\to M$ as in the standard formal theory 
of differential equations\footnote{Jets, prolongations, differential equations 
and their symmetries have direct super-versions, see e.g.\ \cite{HR-MM}.}.

Its infinitesimal symmetry is a point transformation, i.e.\ a vector field on the total space of $TM=J^0\pi$,
whose prolongation to $J^1\pi$ is tangent to $\mathcal{E}$. 
We restrict to the subalgebra $\mathcal{S}_v$ of vertical symmetries, consisting of vector fields tangent to the fibers 
of $\pi:TM\to M$. Such fields are also called gauge symmetries. 

Since the Lie equation $\mathcal{E}$ is linear, its symmetries contain shifts by a solution.
If $\{X_i\}_{i=1}^n$, $n=\dim M$, is a frame on $M$ (for instance, left-invariant fields, corresponding
to a basis of $\m$), and $(x,y)$ are the corresponding coordinates on $TM$ where $T_xM\ni Y=\sum_{i=1}^n y^iX_i$,
then right-invariant fields $R_j=\sum_{i=1}^n a^i_j(x)X_i$ (corresponding to a basis in $\g$)
are symmetries of $\mathcal{E}$ and the infinitesimal shift by the solution $R_j$ is given by the vector field
$\rho_j=\sum_{i=1}^n a^i_j(x)\p_{y^i}$, $1\leq j\leq\dim\g$. 
In addition, the homothety $\sum_{i=1}^n y^i\p_{y^i}$ is in $\mathcal{S}_v$.

We suggest (based on a computation) that those give all infinitesimal symmetries, i.e.\
as a vector space $\mathcal{S}_v$ is isomorphic to $\g\oplus\C$, where $\g$ is $G(3)$ or $F(4)$ respectively, 
and $\C$ is generated by the homothety. 

Note however that $\g\subset\mathcal{S}_v$ is Abelian (shifts commute), 
and the symmetry algebra is its extension by derivation/homothety. 
The Lie superalgebra structure of $\g$ is recovered if we commute the shift symmetries $\rho_j$ with
the Liouville vector field $\sum y^iX^i$ (lifted arbitrarily to $TM$ and then projected down to $M$).

\subsection{Integral submanifolds of higher dimensions}

We call an integral manifold $\varphi:L\to M$ nondegenerate/injective if $\varphi^*$ is surjective on $\mathcal{O}_L$.
Integrality condition means as before $\varphi^*\op{Ann}(\D)=0$. Nondegenerate integral manifolds are ordered
by inclusions according to dimension, but this is only a partial order.

For 
generalized flag varieties $G/P$ of maximal dimension integral curves are maximal integral manifolds:

 \begin{prop}\label{Pr1}
Let $G$ be $G(3)$ or $F(4)$ and $P$ a Borel subgroup. Then maximal integral non-even manifolds of the
left invariant distribution $\D$ on $G/P$, corresponding to $\g_{-1}$, can be at most $(1|1)$-integral curves.
 \end{prop}

Actually, in all cases $P=P^\Xi_{123}$ for $G(3)$ or $P=P^\Xi_{1234}$ for $F(4)$ maximal non-even
integral manifolds are $(1|1)$ curves, except for $P^\Xi_{123}\subset G(3)$ with $\Xi={\rm III,IV}$,
where they are purely odd $(0|1)$ curves. Maximally even integral nondegenerate integral manifolds are $(1|0)$ curves 
except for $P^\Xi_{1234}\subset F(4)$ with $\Xi={\rm I,III,V,VI}$ where they are $(2|0)$ surfaces.
In fact, since $\D$ in this case is generated by root vectors for simple roots, it is not difficult to see that maximal
integrals corresponds to maximally disconnected subgraphs of the Dynkin diagram when it contains only
white and grey nodes (isotropic = black nodes to be treated separately), so Proposition \ref{Pr1}
follows at once from Figures \ref{dynkinG3} and \ref{F4dynkin}.

For other parabolics $P\subset G$ there exist higher dimensional integral submanifolds. 
For brevity of exposition we consider only the case of maximal parabolics,
where such dimensions are clearly biggest possible. Notice that for a concrete superdistribution $\D$ on $G/P$
there may be several maximal dimensions and none largest.
For instance, $\D$ on $G(3)/P^{\rm I}_2$ has $(1|1)$- and $(0|2)$-integral surfaces, but lacks $(1|2)$-integrals.
The argument refers to explicit form of brackets encoded in Tables \ref{negative roots G3} and \ref{negative roots F4}
and the result is as follows.

 \begin{prop}
Dimensions of maximal nondegenerate integral submanifolds of the
left invariant distribution $\D$ on $G/P$, corresponding to $\g_{-1}$, 
for maximal parabolics $P\subset G$, are the following.
 \end{prop}

\vspace{-0.7cm}     
\begin{center}
 \begin{table}[H]
  \begin{minipage}[t]{4cm}
  \[ \begin{array}{|l|c|c|} \hline
G(3)  &  \dim\g_{-1}  &  \mbox{Max dim int}
\vphantom{\dfrac{A}{.}} 
\\ \hline
\fp^{\text{I}}_1	 	&	 (0|7)		&	(0|3) \vphantom{\dfrac{A}{.}} \\[-3pt]
\fp^{\text{I}}_2		&	 (2|2)		&	(1|1), (0|2) \\
\fp^{\text{I}}_3		&	 (4|3)		&	(2|1) \\
\fp^{\text{II}}_1	&	 (2|2)		&	(1|1) \\
\fp^{\text{III}}_1	&	 (4|4)		&	(2|2) \\
\fp^{\text{IV}}_2	&	 (2|4)		&	(1|2) \\
\hline
 \end{array} \]
 \end{minipage}
\hspace{3cm}
 \begin{minipage}[t]{4cm}
  \[  \begin{array}{|l|c|c|} \hline
F(4)  &  \dim\g_{-1}  &  \mbox{Max dim int}
\vphantom{\dfrac{A}{.}} 
\\ \hline
\fp^{\text{I}}_1	 	&	 (0|8)		&	(0|4) \vphantom{\dfrac{A}{.}} \\[-3pt]
\fp^{\text{I}}_2		&	 (3|3)		&	(1|1), (0|3) \\
\fp^{\text{I}}_3		&	 (6|4)		&	(3|2) \\
\fp^{\text{I}}_4		&	 (6|4)		&	(6|4) \\
\fp^{\text{II}}_1	&	 (4|4)		&	(1|3), (2|2) \\
\fp^{\text{III}}_3	&	 (6|4)		&	(3|2) \\
\fp^{\text{IV}}_3	&	 (4|4)		&	(2|2) \\
\fp^{\text{V}}_2	&	 (0|8)		&	(0|4) \\
\fp^{\text{VI}}_2	&	 (0|6)		&	(0|3) \\
\hline
  \end{array} \]
 \end{minipage}
\vspace{0.3cm}
\caption{Maximal integrals for distributions on $G(3)/P_\text{max}$ and $F(4)/P_\text{max}$.}
 \label{Max-int-G3F4}
   \end{table}
\end{center}

\vspace{-1cm}

Note that all maximal parabolics are covered in these tables due to odd reflections isomorphisms shown on
Figures \ref{F4map} and \ref{geometries_diagram}.

\section{Geometries with structure reduction}\label{S7}

In \cite{KST} $G(3)$ was realized as a symmety of the supergeometry $M_2^{\rm IV}$ (super Hilbert-Cartan)
and of the mixed $G(3)$ contact structure $M_1^{\rm IV}$. 
In \cite{ST} $F(4)$ was realized as a symmety of the odd $F(4)$ contact structure $M_1^{\rm I}$ 
and of the mixed $F(4)$ contact structure $M_3^{\rm III}$. 
Here we consider the remaining two special cases, namely $M_1^{\rm I}$ for $G(3)$ and $M_4^{\rm I}$ for $F(4)$,
and explicitly realize those Lie superalgebras by symmetries.

\subsection{Odd $G(3)$ contact structure}\label{SG371}

Consider a supermanifold $M=\C^{1|7}(x|\xi_1,\dots,\xi_7)$ with contact structure given
via contact form 
 $$
\omega=dx-\xi_1\,d\xi_4-\xi_2\,d\xi_5-\xi_3\,d\xi_6-\xi_7\,d\xi_7
 $$ 
with conformal symplectic form (classically a conformal metric) on $\D=\op{Ker}(\omega)$:
 $$
g=-d\omega|_{\D}=d\xi_1\wedge d\xi_4+d\xi_2\wedge d\xi_5+d\xi_3\wedge d\xi_6+d\xi_7\wedge d\xi_7
 $$
The symmetry algebra of this nonholonomic distribution is infinite-dimensional:
$\mathfrak{k}(1|7)$. It is isomorphic to $\mathcal{A}_M$, $f\mapsto X_f\in\op{Vect}(M)$,
equipped with the Lagrange bracket: $[X_f,X_h]=X_{[f,h]}$, where
 $$
X_f=f\p_u-(-1)^{|f|}\Bigl(\hat{\p}_{\xi_4}(f)\p_{\xi_1}+\hat{\p}_{\xi_5}(f)\p_{\xi_2}+\hat{\p}_{\xi_6}(f)\p_{\xi_3}
+\p_{\xi_1}(f)\hat{\p}_{\xi_4}+\p_{\xi_2}(f)\hat{\p}_{\xi_5}+\p_{\xi_3}(f)\hat{\p}_{\xi_6}
+\tfrac12\p_{\xi_7}(f)\hat{\p}_{\xi_7}\Bigr) 
 $$
and $\hat{\p}_{\xi_4}=\p_{\xi_4}-\xi_1\p_u$, $\hat{\p}_{\xi_5}=\p_{\xi_5}-\xi_2\p_u$,
$\hat{\p}_{\xi_6}=\p_{\xi_6}-\xi_3\p_u$, $\hat{\p}_{\xi_7}=\p_{\xi_7}-\xi_7\p_u$.

\smallskip

Consider a supersymmetric (odd) cubic on $\D$ (classically a 3-form):
 $$
q=d\xi_1\,d\xi_4\,d\xi_7+d\xi_2\,d\xi_5\,d\xi_7+d\xi_3\,d\xi_6\,d\xi_7-d\xi_1\,d\xi_2\,d\xi_3+d\xi_4\,d\xi_5\,d\xi_6
 $$
This $q$ reproduces $g$ on $\D$ by the classical relation $g(X,Y)\op{vol}_\D=(i_Xq)(i_Yq)q$,
where $\op{vol}_\D=\prod_1^7d\xi_i$.

The symbol algebra of $\D$ is $\m=\mathfrak{heis}(1|7)$ and $\op{pr}_0(\m)=\mathfrak{co}(7)$,
but the condition that elements $A\in\op{Aut}(\m)_0$ preserve the cubic $A\cdot q=0$ reduces the structure group
$CO(7)$ to $G_0=\C_\times G_2$. In this case $\op{pr}(\m,\g_0)=G(3)$ by Theorem \ref{main},
so by Corollary \ref{cor} the symmetry algebra of $(M,\D,q)$ is $G(3)$.

This gives a realization of $G(3)$ as supersymmetry algebra and realizes it by superalgebra of vector fields
in lowest possible dimension $(1|7)$.

We also note that this implies realization of $G(3)$ as a supersymmetry of a differential equation.
As noted above, integral curves are not enough to encode distributions in general, however for
contact distribution they are sufficient. But in our case we have reduction $(\D,q)$ of the structure group.
To take this into account, consider integral $(0|3)$ submanifolds $\varphi:\C^{0|3}\to M^{1|7}$
such that $\varphi^*\omega=0$, $\varphi^*q=0$. 

Tangents to those form an integral Grassmanian $\mathcal{E}$ inside $\op{Gr}_3(D)$. First of all note that 
the Lagrangian Grassmanian $\Lambda\subset\op{Gr}_3(D)$ of maximally isotropic 3-planes corresponds, classically,
to 3-null planes for a conformal metric structure in 7D, hence dimension\footnote{We count dimensions
over a fixed superpoint in $J^0(\C^{0|3},M)=\C^{0|3}\times M$.} of this quadratic variety is $6$.
$\mathcal{E}\subset\Lambda$ is given by a cubic equation, corresponding to $q$, and its dimension is 5.
We interpret it as a first order PDE $\mathcal{E}'\subset J^1(\C^{0|3},M^{1|7})$. 
Up to reparametrization of the sourse and splitting of the target, the equation is locally represented by
a PDE system $\mathcal{E}\subset J^1(\C^{0|3},\C^{1|4})$. 
It is given by $\op{codim}(\mathcal{E})=12-5=7$ odd nonlinear differential equations.

 \begin{theorem}
The contact symmetry superalgebra of the differential equation $\mathcal{E}$ is $G(3)$.
 \end{theorem}

This follows from the above arguments, 
but also can be computed directly in \textsc{Maple}, where $\mathcal{E}$ is provided explicitly.
Here is the algebra $\mathcal{S}$ given in terms of generating functions $f$ of symmetries 
$X_f\in \op{sym}(\mathcal{E})\simeq G(3)$:
 \begin{multline*}
\mathcal{S}_{\bar0}= \langle 1,\ \xi_1\xi_5,\ \xi_1\xi_6,\ \xi_2\xi_4,\ \xi_2\xi_6,\ \xi_3\xi_4,\ \xi_3\xi_5,\
\xi_1\xi_4-\tfrac23u,\ \xi_2\xi_5-\tfrac23u,\ \xi_3\xi_6-\tfrac23u, \\
\xi_1\xi_2+2\xi_6\xi_7,\ \xi_1\xi_3-2\xi_5\xi_7,\ \xi_2\xi_3+2\xi_4\xi_7,\ 
\xi_4\xi_6-2\xi_2\xi_7,\ \xi_4\xi_5+2\xi_3\xi_7,\ \xi_5\xi_6+2\xi_1\xi_7, \\ 
\tfrac12(\xi_1\xi_2\xi_4\xi_5+\xi_1\xi_3\xi_4\xi_6+\xi_2\xi_3\xi_5\xi_6)
 +(\xi_1\xi_2\xi_3+\xi_4\xi_5\xi_6)\xi_7
+ 2u(\xi_1\xi_4+\xi_2\xi_5+\xi_3\xi_6)-2u^2\rangle;
 \end{multline*}
 
\vskip-25pt
 
 \begin{multline*}
\mathcal{S}_{\bar1}= \langle
\xi_1,\ \xi_2,\ \xi_3,\ \xi_4,\ \xi_5,\ \xi_6,\ \xi_7,\ 
\xi_7(4u-2\xi_1\xi_4-2\xi_2\xi_5-2\xi_3\xi_6)+\xi_1\xi_2\xi_3+\xi_4\xi_5\xi_6\\
\xi_1(4u-\xi_2\xi_5-\xi_3\xi_6)-2\xi_5\xi_6\xi_7,\ 
\xi_2(4u-\xi_1\xi_4-\xi_3\xi_6)+2\xi_4\xi_6\xi_7,\ 
\xi_3(4u-\xi_1\xi_4-\xi_2\xi_5)-2\xi_4\xi_5\xi_7,\\
\xi_4(4u-3\xi_2\xi_5-3\xi_3\xi_6)-2\xi_2\xi_3\xi_7,\ 
\xi_5(4u-3\xi_1\xi_4-3\xi_3\xi_6)+2\xi_1\xi_3\xi_7,\ 
\xi_6(4u-3\xi_1\xi_4-3\xi_2\xi_5)-2\xi_1\xi_2\xi_7\rangle.
 \end{multline*}
 
\subsection{$F(4)$ irreducible geometry}\label{SF464}

Consider the manifold $M=\C^{6|4}$ with empty structure on $\m=T_0M$. Its symmetry corresponding to 
$\op{pr}(\m)=S^\bullet\m^*\otimes\m$ is $\op{Vect}(M)$.
The structure group is $GL(6|4)$ and we impose flat (parallel translated) structure reduction
to $G_0=COSp(2|4)$ with the following Lie superalgebra
 $$
\left[
\begin{array}{cccccc|cccc}
2a_2 &  0 &  0 &  0 &  0 &  0 &  b_8 &  b_7 &  b_5 &  b_6\\
0 &  0 &  a_9 &  2a_6 &  a_4 &  2a_{11} &  b_1 &  b_2 &  b_4 &  b_3\\
0 &  2a_6 &  a_1 + a_3 &  0 &  -a_7 &  -2a_5 &  2b_3 &  2b_4 &  0 &  0\\
0 &  a_9 &  0 &  -a_1 - a_3 &  a_8 &  a_{10} &  0 &  0 &  b_2 &  b_1\\
0 &  2a_{11} &  -a_{10} &  2a_5 &  a_1 - a_3 &  0 &  2b_2 &  0 &  0 &  -2b_4\\
0 &  a_4 &  -a_8 &  a_7 &  0 &  a_3 - a_1 &  0 &  b_1 &  -b_3 &  0\\
\hline
3b_4 &  -b_5 &  -b_7 &  0 &  b_6 &  0 &  a_1 + a_2 &  a_{11} &  a_5 &  a_6\\
-3b_3 &  b_6 &  b_8 &  0 &  0 &  -2b_5 &  a_4 &  a_2 + a_3 &  a_6 &  a_7\\
-3b_1 &  b_8 &  0 &  2b_6 &  0 &  2b_7 &  2a_8 &  a_9 &  a_2-a_1 &  -a_4\\
3b_2 &  -b_7 &  0 &  -2b_5 &  -b_8 &  0 &  a_9 &  a_{10} &  -a_{11} &  a_2 - a_3
\end{array}
\right]+a_{12}\cdot\mathbf{1}_{10}
 $$
embedded into $\mathfrak{gl}(6|4)$. Here $a_i$ are even parameters and $b_j$ are odd parameters,
yielding $(12|8)$ dimensional subalgebra, which is equal to $\g_0$ in the unique $1$-grading of $F(4)$ 
corresponding to $\fp^{\rm I}_4$ (details of this computation are in supplementary \textsc{Maple}).
Then from Corollary \ref{cor} we conclude:

 \begin{theorem}
The symmetry superalgebra of this $G_0$-structure is exactly $F(4)$.
  \end{theorem}
  
The corresponding coordinate realization is given in Appendix \ref{S10}.

 \begin{remark}
Lie superalgebra $\mathfrak{osp}(2|4)$, realized as the supertrace-free part of the above matrix, 
is represented in dimension $(6|4)$. There are two $(6|4)$ ireps of this algebra, dual to each other,
both atypical \cite[p.366]{FSS}. 
 \end{remark}

Geometric structures responsible for this reduction are the following. Decompose 
 $$
\g_0=\mathfrak{z}(\g_0)\oplus\g_0^{ss}=\C\oplus\mathfrak{osp}(2|4),\quad
\g_0^{ss}=\mathfrak{l}=\mathfrak{l}_{\bar0}\oplus\mathfrak{l}_{\bar1}, 
 $$
where $\mathfrak{l}_{\bar0}
=\mathfrak{z}(\mathfrak{l}_{\bar0})\oplus\mathfrak{l}_{\bar0}^{ss}=\mathfrak{so}(2)\oplus\mathfrak{sp}(4)$
and $\mathfrak{l}_{\bar1}=\mathfrak{l}_{-1}\oplus\mathfrak{l}_{+1}=\C_{-}\boxtimes\C^4\oplus\C_{+}\boxtimes\C^4$.
Above $\mathfrak{z}(\g_0)$ is generated by the grading element yielding the decomposition of $F(4)$:
$\g=\g_{-1}\oplus\g_0\oplus\g_1$.

We have $\mathfrak{l}_{\bar0}$-invariant splitting 
$\m=\g_{-1}=(\mathbb{V}^5_0\oplus\mathbb{V}^1_2|\mathbb{V}^4_1)$ with subscripts 
indicating the weights of $\mathfrak{z}(\mathfrak{l}_{\bar0})=\C$, where
with respect to $\mathfrak{l}_{\bar0}^{ss}=C_2$ we have:
$\mathbb{V}^1=[00]=\C$, $\mathbb{V}^4=[10]=\C^4$, $\mathbb{V}^5=[01]=\Lambda^2_0\C^4$.
This splitting is respected by $\mathfrak{l}_{\bar1}$ since 
$\mathfrak{l}_{-1}:\mathbb{V}^1_2\to\mathbb{V}^4_1\to\mathbb{V}^5_0$ and
$\mathfrak{l}_{+1}:\mathbb{V}^5_0\to\mathbb{V}^4_1\to\mathbb{V}^1_2$.

In addition we have $\mathfrak{l}_{\bar0}^{ss}$-invariant tensors
 $$
g\in S^2(\mathbb{V}^5)^*,\ \ 
q\in S^2(\mathbb{V}^1)^*,\ \ 
\omega\in\Lambda^2(\mathbb{V}^4)^*
 $$
of $\mathfrak{z}(\g_0)$ weights 2 each and of $\mathfrak{z}(\mathfrak{l}_{\bar0})$ weights $0,-4,-2$ respectively. 
We also have the following $\mathfrak{l}_{\bar0}$-invariant components of ``multiplication'' $\m\otimes\m\to\m$:
$\Lambda^2\mathbb{V}^4\to\mathbb{V}^1$, $\Lambda^2\mathbb{V}^4\to\mathbb{V}^5\otimes\mathbb{V}^1$, 
$\mathbb{V}^5\otimes\mathbb{V}^1\to\Lambda^2\mathbb{V}^4$.

These tensors and the other relative invariant components, converted to absolute conformal invariants (of weight 0 
with respect to $\mathfrak{z}(\mathfrak{l}_{\bar0})$) via density $q$,
define the structure reduction\footnote{Absolute invariants
have weight 0 wrt $\mathfrak{z}(\g_0)\oplus\mathfrak{z}(\mathfrak{l}_{\bar0})$ and can be derived
via the pair $(g,q)$; for instance $g^{-1}\omega^2q^{-1}$ is such.}. 
Actually, the adjoint representation of $\mathfrak{gl}(6|4)=\m\otimes\m^*$ branched over subalgebra
$\mathfrak{l}=\mathfrak{osp}(2|4)$ is the sum of the trivial representation, small adjoint and 
a unique typical irrep of dimension $(40|40)$, see \cite[Table 3.64]{FSS}, corresponding to dimension split
$(52|48)=(1|0)+(11|8)+(40|40)$:
 $$
\mathfrak{gl}(\m)_{\mathfrak{l}}=\mathbf{1}_{\mathfrak{l}}+\mathfrak{ad}_{\mathfrak{l}}+\mathfrak{R}_{80}. 
 $$
The last component is isomorphic to the cohomology group $H^{0,1}(\m,\g)$ 
for $\g=F(4)$ and $\m$ corresponding to grading via $\fp_4^{\rm I}$.
The projection along the last component is precisely the supersymmetric reduction.

\section{Concluding remarks}\label{S8}

We found realizations of all \(G(3)\) and \(F(4)\) supergeometries (flat parabolic = generalized flag supermanifolds) 
through Tanaka-Weisfeiler prolongations. 
 (Note that realization of a Lie superalgebra $\g$ by an algebra of vector fields can be considered as an instance of Lie's
 third theorem, as this implies local group action via integration; see \cite{GW,MSV,SW,Shc} for discussions of flows.)
In all cases the superalgebras correspond to symmetries of superdistributions or their strucutre reductions.
No higher order reductions was involved. 

The statement is equivalent to vanishing of the Chevaley-Eilienberg cohomology $H^1(\m,\g)_+=0$,
or respectively $H^1(\m,\g)_{\ge0}=0$ for distributions with no reductions, and computing those is a hard task.
In the classical case, this is done via the Kostant theorem \cite{Ko} but in the supercase it is not available.
The technique of spectral sequences and Kostant's version of BBW theorem, exploited in \cite{KST,ST},
is time consuming (computation for 2 cases in \cite{KST} and 2 cases in \cite{ST} occupy many pages) and 
is not reasonable for the remaining 17+53 cases.
Therefore our computations were done explicitly in \textsc{Maple}, using the algorithm in Section \ref{algorithm}. 
Complete results are available as supplementary to the arXiv version of this paper.

Let us remark that a superversion of the Bott-Borel-Weil (BBW) theorem on the sheaf cohomology was addressed 
in the literature \cite{Co}. However the results obtained do not cover representations involved in our work. 
In \cite{V1,V2}, using BBW theorem, global vector fields on generalized flag supermanifolds $G/P$ were computed, 
and in many (nonexceptional) cases they were proven fundamental, i.e.\ isomorphic to $\g=\op{Lie}(G)$. 

The technique used in the proof involves twistor correspondence, i.e.\ bundles $G/P\to G/Q$ with typical fiber $F=Q/P$
for nested parabolics $P\subset Q$. If this fiber $F$ admits no global superfunctions,
i.e.\ $\mathcal{O}(F)|_{F_0}=\C$ (classically by compactness reasons; see Proposition 1 of \cite{V1} addressing 
the super case) then fields from $\op{Vect}(G/P)$ are projectible to $\op{Vect}(G/Q)$ and the homomorphism 
relating these algebras gives a tool to simplify computations. Note that global vector fields assume no invariant
geometry, but a posteriori they preserve the canonical distribution on $G/P$ and a possible structure reduction.

Our work is local\footnote{Global version is also possible due to \cite{O,KST2}.}, but involves nonholonomic geometric structures to restrict the symmetry from the start. 
Twistor correspondences, like those of Figures \ref{geometries_diagram} and \ref{F4map}, 
allow to simplify some computations too.
For instance, the projection of $G(3)$ geometries $M^{\rm IV}_{12}\to M^{\rm IV}_2$ and geometric 
prolongation established in \cite{KST}, allowed to reduce computation of the prolongation 
for the algebra $\m^{\rm IV}_{12}$. We expect more applications of the same idea.
We have not used it in the current work, but observe that it holds true a posteriori.

\appendix

\section{\(G(3)\) and \(F(4)\) superdistributions}\label{S9}

Here we provide dimensions of the gradings of $\m^\Xi_\cch$, or equivalently growth vectors for
the corresponding distributions $\D$ on $M^\Xi_\cch$ for all generalized flag varieties considered in this paper.
$G(3)$ case was presented in \cite{KST}, we repeat it for completeness. For $F(4)$ case 
only maximal parabolics were recorded in \cite{ST} previously.
 
\LTcapwidth=\linewidth
\begin{longtable}{|l|c|c|l|}
\hline
$P^{\Xi}_\cch$ & $\dim M^{\Xi}_\cch$\! & $\mu_{\mathcal{D}}$ & $\mbox{Growth vector }$
\vphantom{$\dfrac{A}{.}$} \\ \hline
$P^{\text{I}}_{1}$	 &	 $(1 | 7)	$ &	 2	 &	 $(0 | 7, 1 | 0)$ \vphantom{$\dfrac{A}{.}$} \\[-3pt]
$P^{\text{I}}_{2}$	 &	 $(6 | 6)	$ &	 4	 &	 $(2 | 2, 1 | 1, 2 | 2, 1 | 1)$ \\
$P^{\text{I}}_{3}$	 &	 $(6 | 5)	$ &	 2	 &	 $(4 | 3, 2 | 2)$ \\
$P^{\text{I}}_{12}$	 &	 $(6 | 7)	$ &	 6	 &	 $(2 | 1, 1 | 2, 2 | 1, 0 | 2, 0 | 1, 1 | 0)$ \\
$P^{\text{I}}_{13}$	 &	 $(6 | 7)	$ &	 4	 &	 $(4 | 2, 1 | 3, 0 | 2, 1 | 0)$ \\
$P^{\text{I}}_{23}$	 &	 $(7 | 6)	$ &	 6	 &	 $(2 | 1, 1 | 1, 1 | 1, 1 | 1, 1 | 1, 1 | 1)$ \\
$P^{\text{I}}_{123}$ &	 $(7 | 7)	$ &	 8	 &	 $(2 | 1, 1 | 1, 1 | 1, 1 | 1, 1 | 1, 0 | 1, 0 | 1, 1 | 0)$ \\
$P^{\text{II}}_{1}$	 &	 $(6 | 5)	$ &	 3	 &	 $(2 | 2, 2 | 2, 2 | 1)$ \\
$P^{\text{II}}_{12}$ &	 $(6 | 7)	$ &	 7	 &	 $(0 | 3, 2 | 0, 0 | 1, 1 | 0, 0 | 2, 3 | 0, 0 | 1)$ \\
$P^{\text{II}}_{13}$ &	 $(7 | 6)	$ &	 5	 &	 $(2 | 2, 1 | 1, 1 | 1, 2 | 1, 1 | 1)$ \\
$P^{\text{II}}_{123}$ &	 $(7 | 7)	$ &	 9	 &	 $(1 | 2, 1 | 1, 1 | 0, 0 | 1, 1 | 0, 0 | 1, 1 | 1, 2 | 0, 0 | 1)$ \\
$P^{\text{III}}_{1}$	 &	 $(5 | 4)	$ &	 2	 &	 $(4 | 4, 1 | 0)$ \\
$P^{\text{III}}_{12}$ &	 $(6 | 7)	$ &	 4	 &	 $(0 | 5, 5 | 0, 0 | 2, 1 | 0)$ \\
$P^{\text{III}}_{13}	$ &	 $(7 | 6)	$ &	 5	 &	 $(2 | 2, 2 | 2, 1 | 1, 1 | 1, 1 | 0)$ \\
$P^{\text{III}}_{123}$ & $(7 | 7)$ &	 7	 &	 $(1 | 2, 1 | 2, 1 | 1, 2 | 0, 1 | 1, 0 | 1, 1 | 0)$ \\
$P^{\text{IV}}_{2}$	 &	 $(5 | 6)	$ &	 3	 &	 $(2 | 4, 1 | 2, 2 | 0)$ \\
$P^{\text{IV}}_{12}$ &	 $(6 | 6)	$ &	 5	 &	 $(2 | 2, 1 | 2, 1 | 2, 1 | 0, 1 | 0)$ \\
$P^{\text{IV}}_{23}	$ &	 $(6 | 7)	$ &	 6	 &	 $(0 | 3, 3 | 0, 0 | 3, 1 | 0, 0 | 1, 2 | 0)$ \\
$P^{\text{IV}}_{123}$ &	 $(7 | 7)$ &	 8	 &	 $(1 | 2, 2 | 1, 1 | 1, 0 | 2, 1 | 0, 0 | 1, 1 | 0, 1 | 0)$ \\
\hline
 \caption*{\vphantom{$\dfrac{A^a}{A}$} \textsc{Table} \ref{G3_distributions_table1}. %
Distributions on generalized flag supermanifolds $M^{\Xi}_\cch=G(3)/P^{\Xi}_\cch$.}
 \label{G3_distributions_table1}
\end{longtable}
 
\vskip-15pt

\LTcapwidth=\linewidth
\begin{longtable}{|l|c|c|l|}
\hline
$P^{\Xi}_\cch$ 	& $\dim M^{\Xi}_\cch\!$  & $\mu_{\mathcal{D}}$ & $\mbox{Growth vector }$
\vphantom{$\dfrac{A}{.}$} 
\\ \hline
$P^{\text{I}}_{1}$		 &	 $(1 | 8)$	 &	 2	 &	 $(0 | 8, 1 | 0)$ \vphantom{$\dfrac{A}{.}$} \\[-3pt]
$P^{\text{I}}_{2}$		 &	 $(7 | 7)$	 &	 3	 &	 $(3 | 3, 3 | 3, 1 | 1)$ \\
$P^{\text{I}}_{3}$		 &	 $(8 | 6)	$ 	 &	 2	 &	 $(6 | 4, 2 | 2)$ \\
$P^{\text{I}}_{4}$		 &	 $(6 | 4)	$ 	 &	 1	 &	 $(6 | 4)$ \\
$P^{\text{I}}_{12}$		 &	 $(7 | 8)	$	 &	 5	 &	 $(3 | 1, 3 | 3, 0 | 3, 0 | 1, 1 | 0)$ \\
$P^{\text{I}}_{13}$		 &	 $(8 | 8)	$	 &	 4	 &	 $(6 | 2, 1 | 4, 0 | 2, 1 | 0)$ \\
$P^{\text{I}}_{14}$		 &	 $(6 | 8)	$	 &	 3	 &	 $(5 | 4, 0 | 4, 1 | 0)$ \\
$P^{\text{I}}_{23}$		 &	 $(9 | 7)	$	 &	 5	 &	 $(3 | 1, 2 | 2, 2 | 2, 1 | 1, 1 | 1)$ \\
$P^{\text{I}}_{24}$		 &	 $(9 | 7)	$	 &	 4	 &	 $(4 | 2, 2 | 2, 2 | 2, 1 | 1)$ \\
$P^{\text{I}}_{34}$		 &	 $(9 | 6)	$	 &	 3	 &	 $(4 | 2, 3 | 2, 2 | 2)$ \\
$P^{\text{I}}_{123}$	 &	 $(9 | 8)	$	 &	 7	 &	 $(3 | 1, 2 | 1, 2 | 2, 1 | 2, 0 | 1, 0 | 1, 1 | 0)$ \\
$P^{\text{I}}_{124}$	 &	 $(9 | 8)	$	 &	 6	 &	 $(4 | 1, 2 | 2, 2 | 2, 0 | 2, 0 | 1, 1 | 0)$ \\
$P^{\text{I}}_{134}$	 &	 $(9 | 8)	$	 &	 5	 &	 $(4 | 2, 3 | 2, 1 | 2, 0 | 2, 1 | 0)$ \\
$P^{\text{I}}_{234}$	 &	 $(10 | 7)$	 &	 6	 &	 $(3 | 1, 2 | 1, 2 | 2, 1 | 1, 1 | 1, 1 | 1)$ \\
$P^{\text{I}}_{1234}$	 &	 $(10 | 8)$	 &	 8	 &	 $(3 | 1, 2 | 1, 2 | 1, 1 | 2, 1 | 1, 0 | 1, 0 | 1, 1 | 0)$ \\
$P^{\text{II}}_{1}$		 &	 $(7 | 5)	$	 &	 2	 &	 $(4 | 4, 3 | 1)$ \\
$P^{\text{II}}_{12}$	 &	 $(7 | 8)	$	 &	 5	 &	 $(0 | 4, 3 | 0, 0 | 3, 4 | 0, 0 | 1)$ \\
$P^{\text{II}}_{13}$	 &	 $(9 | 7)	$	 &	 4	 &	 $(3 | 3, 2 | 2, 3 | 1, 1 | 1)$ \\
$P^{\text{II}}_{14}$	 &	 $(9 | 6)	$	 &	 3	 &	 $(4 | 3, 3 | 2, 2 | 1)$ \\
$P^{\text{II}}_{123}$	 &	 $(9 | 8)	$	 &	 7	 &	 $(2 | 2, 1 | 2, 2 | 0, 0 | 2, 2 | 1, 2 | 0, 0 | 1)$ \\
$P^{\text{II}}_{124}$	 &	 $(9 | 8)	$	 &	 6	 &	 $(2 | 3, 2 | 1, 1 | 1, 1 | 2, 3 | 0, 0 | 1)$ \\
$P^{\text{II}}_{134}$	 &	 $(10 | 7)$	 &	 5	 &	 $(3 | 2, 2 | 2, 2 | 1, 2 | 1, 1 | 1)$ \\
$P^{\text{II}}_{1234}$	 &	 $(10 | 8)$	 &	 8	 &	 $(2 | 2, 2 | 1, 1 | 1, 1 | 1, 1 | 1, 1 | 1, 2 | 0, 0 | 1)$ \\
$P^{\text{III}}_{3}$		 &	 $(7 | 4)	$	 &	 2	 &	 $(6 | 4, 1 | 0)$ \\
$P^{\text{III}}_{13}	$	 &	 $(9 | 6)	$	 &	 3	 &	 $(5 | 4, 3 | 2, 1 | 0)$ \\
$P^{\text{III}}_{23}$	 &	 $(8 | 8)	$	 &	 4	 &	 $(0 | 6, 7 | 0, 0 | 2, 1 | 0)$ \\
$P^{\text{III}}_{34}	$	 &	 $(9 | 6)	$	 &	 4	 &	 $(4 | 3, 2 | 2, 2 | 1, 1 | 0)$ \\
$P^{\text{III}}_{123}$	 &	 $(9 | 8)	$	 &	 5	 &	 $(1 | 4, 3 | 2, 4 | 0, 0 | 2, 1 | 0)$ \\
$P^{\text{III}}_{134}$	 &	 $(10 | 7)$	 &	 5	 &	 $(3 | 3, 3 | 2, 2 | 1, 1 | 1, 1 | 0)$ \\
$P^{\text{III}}_{234}$	 &	 $(9 | 8)	$	 &	 6	 &	 $(1 | 3, 2 | 3, 3 | 0, 2 | 1, 0 | 1, 1 | 0)$ \\
$P^{\text{III}}_{1234}$	 &	 $(10 | 8)$	 &	 7	 &	 $(2 | 2, 1 | 3, 2 | 1, 3 | 0, 1 | 1, 0 | 1, 1 | 0)$ \\
$P^{\text{IV}}_{3}$		 &	 $(8 | 6)$	 &	 3	 &	 $(4 | 4, 2 | 2, 2 | 0)$ \\
$P^{\text{IV}}_{13}$	 &	 $(9 | 7)	$	 &	 4	 &	 $(3 | 3, 3 | 3, 1 | 1, 2 | 0)$ \\
$P^{\text{IV}}_{23}$	 &	 $(9 | 7)	$	 &	 5	 &	 $(3 | 3, 2 | 2, 1 | 1, 1 | 1, 2 | 0)$ \\
$P^{\text{IV}}_{34}$	 &	 $(9 | 6)	$	 &	 5	 &	 $(3 | 2, 2 | 2, 2 | 2, 1 | 0, 1 | 0)$ \\
$P^{\text{IV}}_{123}$	 &	 $(9 | 8)	$	 &	 6	 &	 $(0 | 4, 5 | 0, 0 | 3, 2 | 0, 0 | 1, 2 | 0)$ \\
$P^{\text{IV}}_{134}$	 &	 $(10 | 7)$	 &	 6	 &	 $(3 | 2, 2 | 2, 2 | 2, 1 | 1, 1 | 0, 1 | 0)$ \\
$P^{\text{IV}}_{234}$	 &	 $(10 | 7)$	 &	 7	 &	 $(3 | 2, 2 | 2, 1 | 1, 1 | 1, 1 | 1, 1 | 0, 1 | 0)$ \\
$P^{\text{IV}}_{1234}$	 &	 $(10 | 8)$	 &	 8	 &	 $(1 | 3, 3 | 1, 2 | 1, 0 | 2, 2 | 0, 0 | 1, 1 | 0, 1 | 0)$ \\
$P^{\text{V}}_{2}$		 &	 $(5 | 8)	$	 &	 2	 &	 $(0 | 8, 5 | 0)$ \\
$P^{\text{V}}_{12}$	 &	 $(6 | 8)	$	 &	 3	 &	 $(1 | 4, 0 | 4, 5 | 0)$ \\
$P^{\text{V}}_{23}$	 &	 $(8 | 8)	$	 &	 5	 &	 $(2 | 2, 1 | 4, 2 | 2, 1 | 0, 2 | 0)$ \\
$P^{\text{V}}_{24}$	 &	 $(8 | 8)	$	 &	 4	 &	 $(3 | 4, 1 | 4, 3 | 0, 1 | 0)$ \\
$P^{\text{V}}_{123}$	 &	 $(9 | 8)	$	 &	 6	 &	 $(3 | 1, 1 | 3, 0 | 3, 2 | 1, 1 | 0, 2 | 0)$ \\
$P^{\text{V}}_{124}$	 &	 $(9 | 8)	$	 &	 5	 &	 $(4 | 2, 0 | 4, 1 | 2, 3 | 0, 1 | 0)$ \\
$P^{\text{V}}_{234}$	 &	 $(9 | 8)	$	 &	 7	 &	 $(2 | 2, 1 | 2, 2 | 2, 1 | 2, 1 | 0, 1 | 0, 1 | 0)$ \\
$P^{\text{V}}_{1234}$	 &	 $(10 | 8)$	 &	 8	 &	 $(3 | 1, 1 | 2, 1 | 2, 1 | 2, 1 | 1, 1 | 0, 1 | 0, 1 | 0)$ \\
$P^{\text{VI}}_{2}$		 &	 $(6 | 8)	$	 &	 4	 &	 $(0 | 6, 3 | 0, 0 | 2, 3 | 0)$ \\
$P^{\text{VI}}_{12}$	 &	 $(7 | 8)	$	 &	 6	 &	 $(1 | 3, 0 | 3, 3 | 0, 0 | 1, 0 | 1, 3 | 0)$ \\
$P^{\text{VI}}_{23}$	 &	 $(8 | 8)	$	 &	 7	 &	 $(2 | 2, 0 | 4, 2 | 0, 1 | 0, 0 | 2, 1 | 0, 2 | 0)$ \\
$P^{\text{VI}}_{24}$	 &	 $(8 | 8)	$	 &	 6	 &	 $(2 | 4, 1 | 2, 2 | 0, 0 | 2, 2 | 0, 1 | 0)$ \\
$P^{\text{VI}}_{123}$	 &	 $(9 | 8)	$	 &	 9	 &	 $(3 | 1, 0 | 3, 0 | 2, 2 | 0, 1 | 0, 0 | 1, 0 | 1, 1 | 0, 2 | 0)$ \\
$P^{\text{VI}}_{124}$	 &	 $(9 | 8)	$	 &	 8	 &	 $(3 | 2, 0 | 3, 1 | 1, 2 | 0, 0 | 1, 0 | 1, 2 | 0, 1 | 0)$ \\
$P^{\text{VI}}_{234}$	 &	 $(9 | 8)	$	 &	 9	 &	 $(2 | 2, 1 | 2, 1 | 2, 1 | 0, 1 | 0, 0 | 2, 1 | 0, 1 | 0, 1 | 0)$ \\
$P^{\text{VI}}_{1234}$	 &	 $(10 | 8)$	 &	 11	 &	 $(3 | 1, 1 | 2, 0 | 2, 1 | 1, 1 | 0, 1 | 0, 0 | 1, 0 | 1, 1 | 0, 1 | 0, 1 | 0)$ \\
\hline
 \caption*{\vphantom{$\dfrac{A^a}{A}$} \textsc{Table} \ref{F4_distributions_table2}. %
Distributions on generalized flag supermanifolds $M^{\Xi}_\cch=F(4)/P^{\Xi}_\cch$.}
 \label{F4_distributions_table2}
\end{longtable}

\section{Realization of \(F(4)\) in $(6|4)$ dimensions}\label{S10}

The realization of $F(4)$ from \S\ref{SF464} on $\C^{6|4}(x_i,\xi_j)$ is as follows:
 \[
\g_{-1}=\langle\p_{x_1},\dots,\p_{x_6}\,|\,\p_{\xi_1},\dots,\p_{\xi_4}\rangle,
 \]
 
\vskip-20pt 
 
 \begin{multline*}
(\g_0)_{\bar0}= \langle
x_3\p_{x_3} - x_4\p_{x_4} + x_5\p_{x_5} - x_6\p_{x_6} +\xi_1\p_{\xi_1} - \xi_3\p_{\xi_3}, \
2 x_1 \p_{x_1} + \xi_1 \p_{\xi_1} + \xi_2 \p_{\xi_2} + \xi_3 \p_{\xi_3} + \xi_4 \p_{\xi_4},\\
x_3 \p_{x_3} - x_4 \p_{x_4} - x_5 \p_{x_5} + x_6 \p_{x_6} + \xi_2 \p_{\xi_2} - \xi_4 \p_{\xi_4}, \
x_6 \p_{x_3} - x_4 \p_{x_5} - 2 \xi_3 \p_{\xi_1},\
2 x_5 \p_{x_4} - 2 x_3 \p_{x_6} + \xi_1 \p_{\xi_3},\\
x_4 \p_{x_2} + x_2 \p_{x_3} + \xi_4 \p_{\xi_1} + \xi_3 \p_{\xi_2}, \
x_5 \p_{x_3} - x_4 \p_{x_6} - \xi_4 \p_{\xi_2}, \
2 x_3 \p_{x_2} + 2 x_2 \p_{x_4} + \xi_2 \p_{\xi_3} + \xi_1 \p_{\xi_4}, \\
x_6 \p_{x_2} + x_2 \p_{x_5} + \xi_2 \p_{\xi_1} - \xi_3 \p_{\xi_4}, \
x_6 \p_{x_4} -x_3 \p_{x_5} + \xi_2 \p_{\xi_4}, \
2 x_5 \p_{x_2} + 2 x_2 \p_{x_6} + \xi_1 \p_{\xi_2} - \xi_4 \p_{\xi_3},\\
x_1 \p_{x_1} + x_2 \p_{x_2} + x_3 \p_{x_3} + x_4 \p_{x_4} + x_5 \p_{x_5} + x_6 \p_{x_6}+ \xi_1 \p_{\xi_1} + \xi_2 \p_{\xi_2} + \xi_3 \p_{\xi_3} + \xi_4 \p_{\xi_4}
\rangle,
 \end{multline*}
 
\vskip-20pt 

 \begin{multline*}
(\g_0)_{\bar1}= \langle
3 \xi_1 \p_{x_1} + 2 x_3 \p_{\xi_2} + x_2 \p_{\xi_3} - 2 x_5 \p_{\xi_4}, \
3 \xi_2 \p_{x_1} - 2 x_3 \p_{\xi_1} + x_6 \p_{\xi_3} - x_2 \p_{\xi_4},\\
3 \xi_3 \p_{x_1} - x_2 \p_{\xi_1} - x_6 \p_{\xi_2} - x_4 \p_{\xi_4}, \
3 \xi_4 \p_{x_1} + 2 x_5 \p_{\xi_1} + x_2 \p_{\xi_2} + x_4 \p_{\xi_3}, \
\xi_3 \p_{x_2} + \xi_2 \p_{x_3} - \xi_4 \p_{x_5} + x_1 \p_{\xi_1},\\
\xi_4 \p_{x_2} + \xi_1 \p_{x_3} - 2 \xi_3 \p_{x_6} - x_1 \p_{\xi_2}, \
\xi_1 \p_{x_2} + 2 \xi_4 \p_{x_4} + 2 \xi_2 \p_{x_6} - x_1 \p_{\xi_3}, \
\xi_2 \p_{x_2} + 2 \xi_3 \p_{x_4}  + \xi_1 \p_{x_5} + x_1 \p_{\xi_4}
\rangle,
 \end{multline*}

\vskip-20pt 

 \begin{multline*}
(\g_1)_{\bar0}= \langle
   x_1^2\p_{x_1} + (\xi_1\xi_3 - \xi_2\xi_4) \p_{x_2} + \xi_1\xi_2 \p_{x_3}
- 2\xi_3\xi_4 \p_{x_4} - \xi_1\xi_4 \p_{x_5} + 2\xi_2\xi_3 \p_{x_6}
+ x_1 (\xi_1 \p_{\xi_1} + \xi_2 \p_{\xi_2} + \xi_3 \p_{\xi_3} + \xi_4 \p_{\xi_4}),\\
   3\xi_1\xi_2\p_{x_1} - (x_2^2 - 2x_5x_6)\p_{x_4} 
- 2x_3 (x_2\p_{x_2} +x_3\p_{x_3} +x_5\p_{x_5} +x_6\p_{x_6} +\xi_1\p_{\xi_1} +\xi_2\p_{\xi_2})
- (x_2\xi_2 - x_6\xi_1)\p_{\xi_3} - (x_2\xi_1 - 2x_5\xi_2)\p_{\xi_4},\\
   3\xi_3\xi_4\p_{x_1}  + \tfrac12(x_2^2 - 2x_5x_6)\p_{x_3}
+ x_4 (x_2\p_{x_2} + x_4\p_{x_4} + x_5\p_{x_5} + x_6\p_{x_6} + \xi_3\p_{\xi_3} + \xi_4\p_{\xi_4})
+ (x_2\xi_4 + 2x_5\xi_3)\p_{\xi_1} + (x_2\xi_3 + x_6\xi_4)\p_{\xi_2}, \\
3\xi_1\xi_4\p_{x_1} + (x_2^2 - 2x_3x_4)\p_{x_6} 
+ 2x_5 (x_2\p_{x_2} + x_3\p_{x_3} + x_4\p_{x_4} + x_5\p_{x_5} + \xi_1\p_{\xi_1} + \xi_4\p_{\xi_4})
+ (x_2\xi_1 - 2x_3\xi_4)\p_{\xi_2} - (x_2\xi_4 - x_4\xi_1)\p_{\xi_3},\\
3\xi_2\xi_3\p_{x_1} - \tfrac12(x_2^2 - 2x_3x_4)\p_{x_5} 
- x_6 (x_2\p_{x_2} + x_3\p_{x_3} + x_4\p_{x_4} + x_6\p_{x_6} + \xi_2\p_{\xi_2} + \xi_3\p_{\xi_3})
- (x_2\xi_2 - 2x_3\xi_3)\p_{\xi_1} + (x_2\xi_3 - x_4\xi_2)\p_{\xi_4},\\
3(\xi_2\xi_4-\xi_1\xi_3)\p_{x_1} + 2(x_3x_4+x_5x_6)\p_{x_2} 
+ x_2 (x_2\p_{x_2} + 2x_3\p_{x_3} + 2x_4\p_{x_4} 
+ 2x_5\p_{x_5} + 2x_6\p_{x_6} +\xi_1\p_{\xi_1} +\xi_2\p_{\xi_2} +\xi_3\p_{\xi_3} +\xi_4\p_{\xi_4})\\
+ 2(x_3\xi_4 + x_5\xi_2)\p_{\xi_1} + (2x_3\xi_3 + x_6\xi_1)\p_{\xi_2} + (x_4\xi_2 - x_6\xi_4)\p_{\xi_3} 
+ (x_4\xi_1 - 2x_5\xi_3)\p_{\xi_4}
\rangle,
 \end{multline*}

\vskip-20pt

 \begin{multline*}
(\g_1)_{\bar1}= \langle
   \xi_1 (3x_1\p_{x_1} -x_2\p_{x_2} -2x_3\p_{x_3} -2x_5\p_{x_5} +\xi_2\p_{\xi_2} +2\xi_3\p_{\xi_3} +\xi_4\p_{\xi_4})
- 2(x_3\xi_4 + x_5\xi_2)\p_{x_2} - 2(x_2\xi_4 + 2x_5\xi_3)\p_{x_4} \\
- 2(x_2\xi_2 - 2x_3\xi_3)\p_{x_6} +2x_1x_3\p_{\xi_2} +(x_1x_2 + \xi_2\xi_4)\p_{\xi_3} -2x_1x_5\p_{\xi_4}, \\
   \xi_2 (3x_1\p_{x_1} -x_2\p_{x_2} -2x_3\p_{x_3} - 2x_6\p_{x_6} +\xi_1\p_{\xi_1} +\xi_3\p_{\xi_3} +2\xi_4\p_{\xi_4})
- (2x_3\xi_3 + x_6\xi_1)\p_{x_2} - 2(x_2\xi_3 +x_6\xi_4)\p_{x_4}\\ 
- (x_2\xi_1 - 2x_3\xi_4)\p_{x_5} -2x_1x_3\p_{\xi_1} + x_1x_6\p_{\xi_3} - (x_1x_2 - \xi_1\xi_3)\p_{\xi_4}, \\
   \xi_3 (3x_1\p_{x_1} -x_2\p_{x_2} -2x_4\p_{x_4} - 2x_6\p_{x_6} +2\xi_1\p_{\xi_1} +\xi_2\p_{\xi_2} +\xi_4\p_{\xi_4})
- (x_4\xi_2 - x_6\xi_4)\p_{x_2} - (x_2\xi_2 - x_6\xi_1)\p_{x_3} \\
+ (x_2\xi_4 - x_4\xi_1)\p_{x_5}  - (x_1x_2 +\xi_2\xi_4)\p_{\xi_1}- x_1x_6\p_{\xi_2} - x_1x_4\p_{\xi_4}, \\
   \xi_4 (3x_1\p_{x_1} -x_2\p_{x_2} -2x_4\p_{x_4} -2x_5\p_{x_5} +\xi_1\p_{\xi_1} +2\xi_2\p_{\xi_2} +\xi_3\p_{\xi_3})
- (x_4\xi_1 -2x_5\xi_3)\p_{x_2} - (x_2\xi_1 - 2x_5\xi_2)\p_{x_3} \\
+ 2(x_2\xi_3 -x_4\xi_2)\p_{x_6} + 2x_1x_5\p_{\xi_1} + (x_1x_2 -\xi_1\xi_3)\p_{\xi_2} + x_1x_4\p_{\xi_3}
\rangle.
 \end{multline*}

\end{document}